\documentclass[%
]{mpi2015-cscpreprint}

\usepackage{epstopdf}% To incorporate .eps illustrations using PDFLaTeX, etc.
\usepackage[caption=false]{subfig}% Support for small, `sub' figures and tables
\usepackage{amsmath,graphicx}

\newcommand{\diagdots}[3][-25]{%
\rotatebox{#1}{\makebox[0pt]{\makebox[#2]{\xleaders\hbox{$\cdot$\hskip#3}\hfill\kern0pt}}}%
}
\usepackage{amsfonts}
\usepackage{amssymb}
\usepackage[english]{babel}
\usepackage{pgfplots}
\usepackage{nomencl}
\makenomenclature
\usetikzlibrary{matrix}
\usepackage{parskip}
\usepackage{mymacros}
\usepackage{amsthm}
\usepackage{tikz}
\usepackage{enumerate}
\usetikzlibrary{calc,patterns,decorations.pathmorphing,decorations.markings}

\usepackage{cite}
\usepackage{esvect} % Vektorpfeil
\usepackage{soul} % allows strikethrough text \st
\usepackage{tikz}
\tikzstyle{vertex}=[circle, draw, inner sep=0pt, minimum size=6pt]
\usepackage{graphicx}
\usepackage{float}
\usepackage{mathtools}
\usepackage{todonotes}
\usepackage{booktabs}
\usepackage{ifdraft}
\usepackage{algorithm}
\usepackage{algpseudocode}
\usepackage{setspace}
\usepackage{color}
\usepackage{tikz}
\usetikzlibrary{shapes.geometric, arrows}
\tikzstyle{startstop} = [rectangle, rounded corners, minimum width=3cm, minimum height=0.6cm,text centered, draw=black, fill=blue!20]
\tikzstyle{process} = [rectangle, rounded corners, minimum width=3cm, minimum height=0.6cm, text centered, draw=black, fill=blue!20]
\tikzstyle{errest} = [rectangle, rounded corners, minimum width=3cm, minimum height=0.6cm, text centered, draw=black, fill=blue!20]
\tikzstyle{arrow} = [thick,->,>=stealth]

\renewcommand{\hat}{\widehat}
\renewcommand{\tilde}{\widetilde}

\usepackage[square,sort&compress,comma,numbers]{natbib}% Citation support using natbib.sty
\bibpunct[, ]{[}{]}{,}{n}{,}{,}% Citation support using natbib.sty
\newcounter{mymac@matlab}
\newcommand{\matlab}{MATLAB%
\ifnum\value{mymac@matlab}<1%
\textsuperscript{\textregistered}%
\setcounter{mymac@matlab}{1}%
\fi%
}

\theoremstyle{plain}% Theorem-like structures provided by amsthm.sty
\newtheorem{theorem}{Theorem}[section]
\newtheorem{lemma}[theorem]{Lemma}
\newtheorem{corollary}[theorem]{Corollary}
\newtheorem{proposition}[theorem]{Proposition}

\theoremstyle{definition}
\newtheorem{definition}[theorem]{Definition}

\theoremstyle{remark}

\usepackage{graphicx}

\usepackage{amssymb}
\usepackage{amsthm}

\begin{document}

%%%%%%%%%%%%%%%%%%%%%%%%%%%%%%%%%%%%%%%%%%%%%%%%%%%%%%%%%%%%%%%%%%%%%%%%%%%%%%%%
% PAPER INFORMATION.                                                           %
%%%%%%%%%%%%%%%%%%%%%%%%%%%%%%%%%%%%%%%%%%%%%%%%%%%%%%%%%%%%%%%%%%%%%%%%%%%%%%%%

\title{An adaptive scheme for the optimization of damping positions by decoupling controllability spaces in vibrational systems}

\author[$(\ast)~1$]{J. Przybilla}
\author[$(\ast)~(\ast\ast)~2$]{M. Ugrica Vukojević}
\author[$(\ast\ast)~3$]{N. Truhar}
\author[$(\ast)~4$]{P. Benner}
\affil[$(\ast)$]{Max Planck Institute for Dynamics of Complex Technical Systems, Sandtorstra{\ss}e 1, 39106 Magdeburg, Germany.}
\affil[$(\ast \ast)$]{Department of Mathematics,
University Josip Juraj Strossmayer, Trg Ljudevita Gaja 6, 31000 Osijek, Croatia, }
\affil[1]{\email{przybilla@mpi-magdeburg.mpg.de}, \orcid{0000-0002-8703-8735}}
\affil[2]{\email{ugrica@mpi-magdeburg.mpg.de}, \orcid{0000-0002-8585-8967}}
\affil[3]{\email{ntruhar@mathos.hr}, \orcid{0000-0002-0848-8309}}
\affil[4]{\email{benner@mpi-magdeburg.mpg.de}, \orcid{0000-0003-3362-4103}}

\shorttitle{An adaptive scheme for the optimization of damping positions by decoupling controllability spaces in vibrational systems}
\shortauthor{J. Przybilla, M. Ugrica Vukojević, N. Truhar, P. Benner}
%\shortdate{}

%37C83 Dynamical systems with singularities (billiards,etc.)
%65F10 Iterative numerical methods for linear systems
%65F45 Numerical methods for matrix equations
%65F50 Computational methods for sparse matrices
%65F55 Numerical methods for low-rank matrix approximation; matrix compression
%65M15 Error bounds for initial value and initialboundary value problems involving PDEs
%65N22 Numerical solution of discretized equations for boundary value problems involving PDEs
%70G60 Dynamical systems methods for problems in mechanics
%76D05 Navier-Stokes equations for incompressible viscous fluids
%\msc{41A20, 65F45, 65F55, 65L80, 70G60, 76D05, 93A15}

\abstract{
In this work, the problem of optimizing damper positions in vibrational systems is investigated.
The objective is to determine the positions of external dampers in such a way that the influence of the input on the output is minimized.
The energy response serves as an optimization criterion, whose computation involves solving Lyapunov equations.
Hence, in order to find the best positions, many of these equations need to be solved, and so the minimization process can have a high computational cost.
\\[5pt]
To accelerate the process of finding the optimal positions, we propose a new reduction method.
Our algorithm generates a basis spanning an approximation to the solution space of the Lyapunov equations for all possible positions of the dampers.
We derive an adaptive scheme that generates the reduced solution space by adding the subspaces of interest, and then we define the corresponding reduced optimization problem that is solvable in a reasonable amount of time.
We decouple the solution spaces of the problem to obtain a space that corresponds to the system without external dampers and serves as a starting point for the reduction of the optimization problem.
In addition, we derive spaces corresponding to the different damper positions that are used to expand the reduced basis if needed.
To evaluate the quality of the basis, we introduce an error indicator based on the space decomposition.
Our new technique produces a reduced optimization problem of significantly smaller dimension that is faster to solve than the original problem, which we illustrate with some numerical examples.
}
\novelty{We propose a new method that reduces the optimization problem of finding the suboptimal damping positions in mechanical systems.
For that, we invent a procedure that decouples the controllability spaces into parameter-independent and parameter-dependent subspaces.
We derive an adaptive reduced basis method that uses these subspaces to define a reduced optimization problem and hence accelerates the optimization process.
Also a new error indicator is invented that is used to evaluate the quality of the approximations.
The performance of our new method is illustrated in several examples.}

\maketitle

%%%%%%%%%%%%%%%%%%%%%%%%%%%%%%%%%%%%%%%%%%%%%%%%%%%%%%%%%%%%%%%%%%%%%%%%%%%%%%%%
% Introduction                                                               %
%%%%%%%%%%%%%%%%%%%%%%%%%%%%%%%%%%%%%%%%%%%%%%%%%%%%%%%%%%%%%%%%%%%%%%%%%%%%%%%%

\section{Introduction}\label{sec:Intro}
Structures such as buildings, bridges or damps are often disturbed by external forces, e.g. wind occurrences, earthquakes or pedestrian movements. These disturbances can cause vibrations, deflection or even damages in the constructions which can be prevented by adding external dampers.
This paper is about determining the best positions for dampers in order to minimize the effects of external forces on the structure.
We consider vibrational systems of the form
\begin{align}\label{eq:SOsys}
\begin{split}
\bM\ddot{\bx}(t) + \bD(c, g)\dot{\bx}(t) + \bK\bx(t) &= \bB\bu(t),\\
\by(t) &= \bC \bx(t),
\end{split}
\end{align}
where $\bM,~\bD(c, g),~\bK\in\Rnn$ are the mass matrix, the damping matrix, and the stiffness matrix, respectively,
and where the damping matrix depends on parameters $c\in\bcD_{c}$ and $g\in\bcD_{g}$.
Moreover, the matrix $\bB\in\Rnm$ is the input matrix and $\bC\in\Rpn$ is the output matrix.
The vectors $\bu(t)\in\Rm$, $\bx(t)\in\Rn$ and $\by(t)\in\Rp$ are respectively representing the input, the state and the output of the system.
Naturally, the matrices $\bM$, and $\bK$ are symmetric and positive definite, while $\bD(c, g)$ is symmetric and generally positive semi definite matrix.

In our context, the damping matrix $\bD(c,g)$ is symmetric, positive definite, consisting of two parts: a parameter-independent internal damping $\bD_{\intern}$ which is positive definite and a parameter-dependent external damping $\bD_{\extern}(c,g)$, i.e.,
\[
\bD(c,g) = \bD_{\intern} + \bD_{\extern}(c,g).
\]

The external damping $\bD_{\extern}(c,g)$ depends on two types of parameters.
The first ones are the damping gains
$g = \begin{bmatrix}
g_1,\dots , g_{\ell}
\end{bmatrix}\in\bcD_g\subset\R^{\ell}_+$,
which represent the viscosities of the dampers.
We assume that the viscosities $g_j$ are fixed over time and lie in given intervals $[g_j^-, \, g_j^+]$, for all $j = 1,\dots, \ell$.
We collect these bounds by setting $g\in\bcD_g$, where the parameter set $\bcD_g$ contains all given conditions.
The second set of parameters consists of the damping positions $c\in\bcD_c\subset\{1,\dots,n\}^{\ell}$ which are stored in the matrix $\bF(c)\in\R^{n\times\ell}$.
The structure of the matrix $\bF(c)$ depends on the damper type so that e.g. grounded dampers are described by unit vectors that are concatenated to build the matrix $\bF(c)$.
This is described in more detail for the different numerical examples.
The resulting external damper is then given as
\begin{align*}
\bD_{\extern}(c, g):=\bF(c)\bG(g)\bF(c)^\T,\qquad \bG(g) := \diag{g_1, \dots ,g_{\ell}}.
\end{align*}
In addition to the external dampers, the system is internally damped.
There are several different models for internal damping.
In this work we use a small multiple of the critical damping that is
\begin{align}\label{eq:intDamp}
\bD_{\intern} := 2\alpha \bM^{\frac{1}{2}}\left(\bM^{-\frac{1}{2}}\bK\bM^{-\frac{1}{2}}\right)^{\frac{1}{2}}\bM^{\frac{1}{2}},
\end{align}
for $\alpha\ll 1$, which is a widely used convention and has been applied, for example, in \cite{morBenTT11, morBenTT13}.
However, our theory can be applied to all modal dampings, which includes, for example, Rayleigh damping defined in \cite{Wil04,KuzTT12}.
We assume that the number of external dampers $\ell$ is significantly smaller than the dimension $n$.
Because of this structure, these systems are also asymptotically stable, i.e. all eigenvalues $\lambda(c,g)$ of the polynomial eigenvalue problem $(\lambda(c,g)^2\bM + \lambda(c,g)\bD(c,g) + \bK)x(c,g) = 0$ have a negative real part.
%In \cite{BlaCGetal12}, the authors describe semi-actively damped systems and their resulting structure in more detail.

In order to simplify the computations in this work, we consider a system transformation, that was introduced, e.g. in \cite{TruV07}, where it is shown that there exists a transformation $\bPhi$ such that
\begin{align*}
\bPhi^\T \bM\bPhi = \bI_n,\qquad \bPhi^\T \bK\bPhi = \bOmega^2 = \diag{\omega_1^2,\dots, \omega_n^2}, \qquad\bPhi^\T \bD_\intern\bPhi = 2\alpha \bOmega,
\end{align*}
i.e. the matrix $\bPhi$ simultaneously diagonalizes the mass matrix, the stiffness matrix and the internal damping where the transformed mass matrix is the identity matrix.
The values $\omega_1, \dots, \omega_{n}$ are the eigenvalues of the undamped system and are called \emph{eigenfrequencies}, see \cite{TruV09}.
The transformed external damping matrix consists of low-rank factors \[\bPhi^\T \bD_{\extern}(c, g)\bPhi = \widetilde{\bF}(c) \bG(g)\widetilde{\bF}(c)^\T,\qquad\widetilde{\bF}(c): =\bPhi^\T \bF(c).\]
We summarize the transformed damping matrix by defining $\widetilde{\bD}(c, g):= 2\alpha \bOmega + \widetilde{\bF}(c) \bG(g)\widetilde{\bF}(c)^\T$.
Moreover, we define $\widetilde{\bB}:=\bPhi^\T \bB$ and $\widetilde{\bC}:= \bC\bPhi$ to obtain the transformed system
\begin{align}\label{eq:SOsysPhi}
\begin{split}
\ddot{\widetilde{\bx}}(t) + \widetilde{\bD}(c, g)\dot{\widetilde{\bx}}(t) + \bOmega^2\widetilde{\bx}(t) &= \widetilde{\bB}\bu(t),\\
\by(t) &= \widetilde{\bC} \widetilde{\bx}(t),
\end{split}
\end{align}
which is equivalent to the original system \eqref{eq:SOsys} in the sense that both systems have the same input-output behavior.
%Hence, our goal is to find the damping values $g\in\bcD_g$ and positions $c\in\bcD_c$ that minimize the output in an appropriate norm.

Vibrational systems and the optimization of their contained dampers have been studied in a large number of papers in the last decades, for example in \cite{Bea96,DuX10,Gen09,Inm06,KarVL94,Ves11,MulW21,LiuRZ20,JakMTetal22}, where external dampers are considered in systems that already contain internal damping of small magnitude.
The aim is to choose external dampers that stabilize the system and shift eigenfrequencies in such a way that possible external loads do not lead to resonances.
There are two variables that can be optimized: the damper positions and the damping gains.
Moreover, depending on the application, different criteria are chosen to quantify the stability of systems and the response to external disturbances.
If homogeneous systems are considered, i.e. systems \eqref{eq:SOsys} with $\bB\equiv 0$, then the spectral abscissa or the total average energy are used. Examples are described in \cite{morNak03,FreL99,morBenTT13,morTruTP17,TruV09}.\\
In the non-homogeneous case, i.e. $\bB\neq 0$, as considered in this work, external disturbances are taken into account, which potentially play an important role in real-life scenarios.
In these cases, the average displacement amplitude can be evaluated, which minimizes the square of the norm of the displacement $\widetilde{\bx}(t)$ averaged over a certain time period, see \cite{KuzTT16,TruTV15}.
Another criterion, which is used in this work, is the average energy amplitude corresponding to the minimization of the \emph{system response}, i.e.
\begin{align*}
\bJ(c,g):=\|\mathcal{\bcG}(\cdot;c,g)\|_{\mathcal{H}_2} = \left( \frac{1}{2\pi} \int_{-\infty}^{\infty} \mathrm{trace}(\mathcal{\bcG}(\mathrm{i}w;c,g)^{\rmH}\mathcal{\bcG}(\mathrm{i}w;c,g))\mathrm{d}w \right)^{\frac{1}{2}},
\end{align*}
where $\bcG(s; c, g):=\widetilde{\bC}(s^2I + s\widetilde{\bD}(c, g) + \bOmega^2)^{-1}\widetilde{\bB}$ is the \emph{transfer function} of the system describing the input-to-output behavior in the frequency domain.
This optimization criterion was also used in \cite{morBenKTetal16,morTomBG18}.
We choose this particular criterion since we want to minimize the maximum deflections, more precisely, the maximum time response $\max_{t\geq0}\|\by(t)\|_{\infty}$, and therefore consider the $L_{\infty}$-norm of the output that satisfies the bound
\[
\|\by\|_{L_{\infty}}\leq \|\bcG(\cdot; c, g)\|_{\cH_2}\| \bu\|_{L_2}.
\]
Alternatively, in \cite{BlaCGetal12,morTomV20,morPrzPB23}, the authors optimize the $\cH_{\infty}$-norm of the systems constraining the $L_2$ norm of the output $\by$ of the corresponding system, which can be interpreted as the worst-case amplification of the energy in the performance output by an input signal.

Most of the established methods consider the optimization of the damping gains.
The optimization of the discrete damper positions is still a difficult problem, especially for large systems, which has been studied in \cite{GurM92,Tak98,DymD21,morTruTP17a}.
In particular, in \cite{morBenTT11,morTom11} the authors describe the optimization using a discrete-to-continuous approach, which is modified and used in this paper.

In order to simplify the computation of the system response $\bJ(c,g)$, we transform the second-order system \eqref{eq:SOsysPhi} into a first-order system
\begin{align}\label{eq:FOsys}
\begin{split}
\dot{\bz}(t) &= \bcA(c,g)\bz(t) + \bcB \bu(t),\\
\by(t) &= \bcC \bz(t)
\end{split}
\end{align}
with corresponding matrices
\begin{align}\label{eq:FOMatr}
\bcA(c,g):=\begin{bmatrix}
0 & \bI_n \\
-\bOmega^2 & -\widetilde{\bD}(c,g)
\end{bmatrix}, \quad \bcB := \begin{bmatrix}
0\\
\widetilde{\bB}
\end{bmatrix}, \quad \bcC := \begin{bmatrix}
\widetilde{\bC} & 0
\end{bmatrix}\;\text{  and  }\; \bz(t) := \begin{bmatrix}
\widetilde{\bx}(t)\\
\dot{\widetilde{\bx}}(t)
\end{bmatrix}.
\end{align}
As described in \cite{ZhoDG96}, the $\cH_2$-norm of the transfer function can be computed as
\begin{align}\label{eq:SysResTr}
\bJ(c,g)=\sqrt{\trace{\bcC \bcP(c,g) \bcC^\T}},
\end{align}
where the matrix $\bcP(c,g)$ is the \emph{controllability Gramian} of \eqref{eq:FOsys} and solves the Lyapunov equation
\begin{align}\label{eq:Lyap}
\bcA(c,g)\bcP(c,g) + \bcP(c,g)\bcA(c,g)^\T = -\bcB\bcB^\T.
\end{align}

In the following, we determine damping positions $c\in\bcD_c$ and damping gains $g\in\bcD_g$ such that the system response described in \eqref{eq:SysResTr} is locally minimal.
Since the computation of the system response includes solving a Lyapunov equation that is given in \eqref{eq:Lyap}, the optimization process of the external dampers leads to high computational costs or is even unfeasible if the system is of too large dimension.
Hence, we apply some reduction techniques to reduce the optimization problem such that it becomes solvable in a reasonable time.
To that aim, we apply a Galerkin scheme, i.e. we determine bases $\bW, ~\bT\in\R^{n, r}$ with $\bW^{\T}\bT = \bI_r$ that generate the reduced matrices
\begin{align}\label{eq:redMatWT}
\bD_{\rr}(c, g):= \bW^{\T}\widetilde{\bD}(c, g)\bT,\quad \bK_{\rr}:= \bW^{\T}\bOmega^2\bT,\quad \bB_{\rr}:= \bW^{\T}\widetilde{\bB},\quad \bC_{\rr}:= \widetilde{\bC}\bT
\end{align}
and the corresponding reduced system
\begin{align}\label{eq:SOsys_red}
\begin{split}
\ddot{\bx}_{\rr}(t) + \bD_{\rr}(c, g)\dot{\bx}_{\rr}(t) + \bK_{\rr}\bx_{\rr}(t) &= \bB_{\rr}\bu(t),\\
\by_{\rr}(t) &= \bC_{\rr} \bx_{\rr}(t).
\end{split}
\end{align}
The goal is to find matrices $\bW,~\bT$ so that the reduced system \eqref{eq:SOsys_red} approximates the input-to-output behavior of the system \eqref{eq:SOsysPhi}, i.e., the error $\|\by-\by_{\rr}\|$ is small in a suitable norm for all parameters $c\in\bcD_c$ and $g\in\bcD_g$.
Then the external damping can be optimized for system \eqref{eq:SOsys_red} instead of system \eqref{eq:SOsysPhi} which will save computational costs.
The system response is a highly non-convex function, and hence has several local minima.
The aim in this work is to find the same local minimizer as for the full-order model.

The paper is structured as follows:
First, in Section \ref{sec:Preliminaries}, we review two different reduction methods for parameter-independent systems.
Then, in Section \ref{sec:SatSpRep}, we derive a decoupling of the controllability space, which is then used in Section \ref{sec:Redu} to generate the reduced systems.
First, in Subsection \ref{ssec:BT_R}, we generate the projection spaces using the Gramians of the systems spanning the controllability spaces for different external dampers.
Then, in Subsection \ref{ssec:IRKA_R}, the reduced projection spaces are constructed using the iterative rational Krylov algorithm (IRKA) method that was introduced for second-order systems in \cite{morTomBG18}.
In Section \ref{sec:ErrorDet}, we describe an error expression that describes the quality of the controllability space approximation, which can be used to improve the introduced methods.
Finally, in Sections \ref{sec:CompDet} and \ref{sec:NumRes}, we explain some details about the implementation of the methods and the numerical results for some selected examples.

\section{Preliminaries}\label{sec:Preliminaries}

In this section, we consider the reduction of the second-order system \eqref{eq:SOsysPhi} for fixed external dampers described by some $c\in\bcD_c$, $g\in\bcD_g$, and hence, we consider the parameter-independent case, i.e. $\widetilde{\bD}\equiv \widetilde{\bD}(c,g)$ in this section.
Later, in Section \ref{sec:SatSpRep} and the following ones, we will consider the case of several damping parameters $c$ and $g$ that are to optimize.
%\todo[inline]{add something to smooth things}
We review two reduction methods for second-order systems that will be used later in this work to derive a reduced optimization method.
In Subsection \ref{ssec:BT}, we explain balanced truncation for second-order systems, and afterwards, in Subsection \ref{ssec:IRKA}, we describe an IRKA method that is specified for second-order systems.

\subsection{Balanced truncation}\label{ssec:BT}
In this section, we describe balanced truncation suitable for second-order systems as presented in \cite{morMeyS96, morChaLVetal06, morChaGVetal05, morReiS08}.
We assume that the system \eqref{eq:SOsysPhi} has homogeneous initial conditions, i.e. $\widetilde{\bx}(0)=0$ and $\dot{\widetilde{\bx}}(0)=0$.
For the inhomogeneous case, the authors in \cite{morPrzPB22} derive two balanced truncation methods that follow a similar methodology as presented in this section.

First, we derive some input-to-state and state-to-output mappings that are used to derive suitable second-order Gramians.
To do so, we consider the (parameter-independent) transfer function
\[
\bcG(s):=\widetilde{\bC}(s^2 + s\widetilde{\bD} + \bOmega^2)^{-1}\widetilde{\bB}
\]
that we generate by applying the Laplace transform to the second-order system \eqref{eq:SOsysPhi} with zero initial conditions.
From that, we can derive the two mappings that are introduced in the following definition.
\begin{definition}\label{def:SO_mappingsGramians}
The \emph{input-to-state mapping} $\boldsymbol{\cR}_{\SO}$ and the \emph{state-to-output mapping} $\boldsymbol{\cS}_{\SO}$ of the asymptotically stable second-order system \eqref{eq:SOsysPhi} with zero initial conditions are
\begin{align*}
\boldsymbol{\cR}_{\SO}(s) := \left( s^2\bI +s\widetilde{\bD} + \bOmega^2 \right)^{-1}\widetilde{\bB},\qquad
\boldsymbol{\cS}_{\SO}(s) :=\widetilde{\bC}\left( s^2\bI +s\widetilde{\bD} + \bOmega^2\right)^{-1}.
\end{align*}
The corresponding \emph{second-order controllability Gramian} $\bP_{\SO}$ and \emph{observability Gramian} $\bQ_{\SO}$ are defined as

\begin{align}\label{eq:SO_Gram}
\bP_{\SO}:= \frac{1}{2\pi}\int_{-\infty}^{\infty} \boldsymbol{\cR}_{\SO}(iw)\boldsymbol{\cR}_{\SO}(-iw)^{\T}\mathrm{d}w,\qquad
\bQ_{\SO}:=\frac{1}{2\pi}\int_{-\infty}^{\infty} \boldsymbol{\cS}_{\SO}(-iw)^{\T}\boldsymbol{\cS}_{\SO}(iw)\mathrm{d}w.
\end{align}
\end{definition}
The matrix $\bP_{\SO}$ is also called \emph{position controllability Gramian} and $\bQ_{\SO}$ is called \emph{velocity observability Gramian}.
We want to emphasize that $\bP_{\SO}$ spans the controllability space and $\bQ_{\SO}$ the observability space of system \eqref{eq:SOsysPhi}.
%The first-order system \eqref{eq:FOsys} is then equivalent to the second-order system so that the corresponding transfer function can also be written in the first order representation $\bcG(s) = \boldsymbol{\cC}\left( \boldsymbol{\cA} -s\bI \right)^{-1}\boldsymbol{\cB}$.
In order to determine the second-order controllability Gramian and observability Gramian of system \eqref{eq:SOsysPhi}, we define the controllability Gramian and the observability Gramian of the respective first-order system \eqref{eq:FOsys} that are
\begin{subequations}\label{eq:FO_Gram}
\begin{align}
\boldsymbol{\cP}:=&{}\begin{bmatrix}
\bP_1 & \bP_2 \\
\bP_2^{\T} & \bP_3
\end{bmatrix} := \frac{1}{2\pi}\int_{-\infty}^{\infty}(iw\bI-\boldsymbol{\cA})^{-1}\boldsymbol{\cB}\boldsymbol{\cB}^{\T}(-iw\bI-\boldsymbol{\cA})^{-\T}\mathrm{d}w,\label{eq:FO_Gram_a} \\
\boldsymbol{\cQ}:=&{}\begin{bmatrix}
\bQ_1 & \bQ_2\\\
\bQ_2^{\T} & \bQ_3
\end{bmatrix} := \frac{1}{2\pi}\int_{-\infty}^{\infty}(-iw\bI-\boldsymbol{\cA})^{-\T}\boldsymbol{\cC}^{\T}\boldsymbol{\cC}(iw\bI-\boldsymbol{\cA})^{-1}\mathrm{d}w.\label{eq:FO_Gram_b}
\end{align}
\end{subequations}
The first order Gramians are computed by solving the Lyapunov equations
\[
 \bcA\boldsymbol{\cP} + \boldsymbol{\cP}\bcA^{\T} = -\bcB\bcB^{\T}, \qquad
  \bcA^{\T}\boldsymbol{\cQ} + \boldsymbol{\cQ}\bcA = -\bcC^{\T}\bcC.
\]
Based on these first-order Gramians, the second-order Gramians can be computed as described in the following proposition from \cite{morPrzPB22,morReiS07}.
\begin{proposition}\label{prop:SO_Gram}
Consider a second-order system as described in \eqref{eq:SOsysPhi}, that is asymptotically stable.
The corresponding second-order controllability Gramian $\bP_{\SO}$ and observability Gramian $\bQ_{\SO}$ as defined in \eqref{eq:SO_Gram} satisfy
\[
\bP_{\SO} = \bP_1 \qquad \text{ and } \qquad \bQ_{\SO} = \bQ_3,
\]
where $\bP_1$ is the upper-left block of the first-order controllability Gramian $\bcP$, and $\bQ_3$ is the lower-right block of the first-order observability Gramian $\bcQ$ as defined in \eqref{eq:FO_Gram}.
\end{proposition}

In order to derive projection matrices that reduce the system matrices as described in \eqref{eq:redMatWT}, we derive some low-rank factors or Cholesky factors $\bR_1$ and $\bS_3$ with $\bP_1 = \bR_1\bR_1^{\T}$ and $\bQ_3 = \bS_3\bS_3^{\T}$.
Those factors exist since $\bP_1$ and $\bQ_3$ are symmetric and positive semi-definite matrices.
They are then used to compute the singular value decomposition
$
\bS_3^{\T}\bR_1 = \bU\boldsymbol{\Sigma}\bX^{\T}
$
and to derive the balancing and truncating projection matrices that are
\begin{equation}\label{eq:SOprojection_matrices}
\bW := \bS_3 \bU_{\rr} \boldsymbol{\Sigma}_{\rr}^{-\frac{1}{2}}, \qquad \bT := \bR_1 \bX_{\rr}\boldsymbol{\Sigma}_{\rr}^{-\frac{1}{2}},
\end{equation}
where $\boldsymbol{\Sigma}_{\rr}$ is the diagonal matrix containing the $r$ largest singular values of $\boldsymbol{\Sigma}$.
Moreover, $\bU_{\rr}$ and $\bX_{\rr}$ include the $r$ leading columns of $\bU$ and $\bX$.
Projecting by $\bW$ and $\bT$ provides the reduced system \eqref{eq:SOsys_red}.

\subsection{Iterative rational Krylov method}\label{ssec:IRKA}
In this section, we briefly present the iterative rational Krylov method (IRKA) suitable for the case of second-order systems, as introduced in \cite{morTomBG18}, where we aim to find a reduced system as in \eqref{eq:SOsys_red} with a transfer function
\begin{align}\label{eq:SO_TF}
\bcG_{\rr}(s):=\bC_{\rr}(s^2\bM_{\rr} + s\bD_{\rr} +\bK_{\rr})^{-1}\bB_{\rr}
\end{align} and matrices as in \eqref{eq:redMatWT}.
Within the IRKA approach, we determine a reduced system that satisfies the left, right and bi-tangential interpolation conditions
\begin{align}\label{eq:interpol_cond}
\bcG(-s_i)b_i &= \bcG_{\rr}(-s_i)b_i, \qquad
c_i^{\rmH}\bcG(-s_i) = c_i^{\rmH}\bcG_{\rr}(-s_i), \qquad
c_i^{\rmH}\bcG'(-s_i)b_i = c_i^{\rmH}\bcG_{\rr}'(-s_i)b_i
\end{align}
for interpolation points $s_1, \dots, s_r$, and tangential directions $b_{1}, \dots, b_{r}\in\R^{m}$ and $c_{1}, \dots, c_{r}\in\R^{q}$.
The transfer function $\bcG_{\rr}(s)$ can also be written in terms of their poles $\lambda_k$ and residues $\widetilde{c}_k\widetilde{b}_k^{\T}$ as
\begin{align}\label{eq:G_pole}
\bcG_{\rr}(s) = \sum_{k=1}^{r} \frac{\widetilde{c}_k\widetilde{b}_k^{\T}}{s-\lambda_k}.
\end{align}
The authors in \cite{morGugAB08} show that if a transfer function $\bcG_{\rr}$ of the form in \eqref{eq:G_pole} is a local minimizer of
\begin{align}\label{eq:H2opti}
\|\bcG(\cdot)-\bcG_{\rr}(\cdot)\|_{\cH_2} = \min_{\substack{\dim(\widehat{\bcG}(\cdot)) = r \\ \widehat{\bcG}(\cdot) \text{ is stable }}} \|\bcG(\cdot)-\widehat{\bcG}(\cdot)\|_{\cH_2},
\end{align}
then $\bcG_{\rr}$ is a tangential Hermite interpolant of $\bcG$ at $-\lambda_k$, $\widetilde{c}_k$, and $\widetilde{b}_k$, $k=1,\dots,r$, i.e. the conditions in \eqref{eq:interpol_cond} are satisfied for these interpolation points and tangential directions.
The IRKA method applies a Newton method that derives in each iteration new interpolation points and tangential directions that converge to the optimal ones, and define a reduced system with optimal transfer function $\bcG_{\rr}$.
%
%The authors of \cite{morTomBG18} apply first-order necessary conditions for $\bcG_{\rr}(s;c,g)$ to be $\cH_2$-optimal as described in \cite{MeiL67}, that require that $\bcG_{\rr}(s;c,g)$ is a Hermite interpolant to the full-order system at certain points on the complex plane.
%For MIMO systems we consider these properties in particular directions.

Defining and computing the projecting matrices $\bW$ and $\bT$ that define the reduced matrices \eqref{eq:redMatWT} and the system \eqref{eq:SOsys_red} that fulfill the properties in \eqref{eq:interpol_cond} was done in \cite{morAntBG10,morGugAB08}.
However, the choice of the projecting bases also depends on additional conditions that are applied to the reduced systems.
For mechanical systems \eqref{eq:SOsysPhi}, the aim is to find bases that preserve the symmetry and the positive definiteness of the mass matrix, the damping matrix, and the stiffness matrix in order to obtain an asymptotically stable reduced system. Hence, we set $\bW = \bT$.
Also, the reduced system is supposed to be of second-order structure.
Hence, the sym2IRKA method, presented in \cite{morTomBG18}, generates a reduced transfer function of the structure shown in \eqref{eq:SO_TF} that represents a second-order system as in \eqref{eq:SOsys_red}.
For that it uses a one-sided projection approach that generates a basis matrix $\bW$ with
\begin{align*}
\range{\bW}\supseteq\myspan{(s_1^2\bI_{n} + s_1\widetilde{\bD} + \bOmega^2)^{-1}\widetilde{\bB} b_1,\dots, (s_r^2\bI_{n} + s_r\widetilde{\bD} + \bOmega^2)^{-1}\widetilde{\bB} b_{r}},
\end{align*}
for interpolation points $s_1, \dots, s_{r}$ and directions $b_1, \dots, b_{r}$ and defines the matrices \eqref{eq:redMatWT} for $\bT=\bW$.
The resulting reduced matrices generate a transfer function of the form \eqref{eq:SO_TF} that interpolates the original transfer function $\bcG$ using $r$ interpolation points and tangential direction.
Within the sym2IRKA method those are updated in each step to converge towards the optimal ones.

This procedure results in Algorithm \ref{algo:IRKA_SO}.
After a basis $\bW$ is built in Step \ref{StepW}, the reduced matrices are defined in Step \ref{StepMatr} which define a reduced second-order system with a transfer function of order $2r$.
Since this order is twice the dimension we aim for, an internal truncation step is applied.
By applying a second IRKA method to the first-order representation of the reduced system corresponding to the reduced matrices, we obtain a system of dimension $r$ with a transfer function $\bcG_{\rr}$.
That way, we obtain the transfer function $\bcG_{\rr}(s)$ of the form \eqref{eq:G_pole}.
We solve the resulting reduced eigenvalue problem to obtain the interpolation points and tangential directions.
Those are then used to derive the basis $\bW$ that provides for the consecutive reduced system.

\begin{algorithm}[tb]
\caption{sym2IRKA \cite{morTomBG18}}
\label{algo:IRKA_SO}
\begin{algorithmic}[1]
\Require{The original system \eqref{eq:SOsysPhi}, maximum number of iterations $N_{\max}$.}
\Ensure{A reduced system \eqref{eq:SOsys_red} that satisfies \eqref{eq:H2opti}.}
\State Choose initial expansion points $\bcS=\{s_1, \dots, s_{r}\}$ and right tangential directions $b_1, \dots, b_{r}$.
\While{iteartion number $\leq N_{\max}$ and $\bcS$ did not converge}
\State Set	
\[
\bW = \begin{bmatrix}(s_1^2\bI_{n} + s_1\widetilde{\bD} + \bOmega^2)^{-1}\widetilde{\bB} b_1 & \dots &  (s_r^2\bI_{n} + s_r\widetilde{\bD} + \bOmega^2)^{-1}\widetilde{\bB} b_{r}
\end{bmatrix}.
\]\label{StepW}
\State Determine reduced matrices
\[
\bD_{\rr}:=  \bW^{\T}\widetilde{\bD}\bW, \quad
\bK_{\rr}:=  \bW^{\T}\bOmega^2\bW,  \quad
\bB_{\rr}:= \bW^{\T}\widetilde{\bB}, \quad
\bC_{\rr}:=  \widetilde{\bC}\bW.
\] \label{StepMatr}
\State Reduce the resulting first-order system with matrices \eqref{eq:FOMatr} to an order $r$ system with a transfer function
$
\bcG_{\rr}(s)=\sum_{k=1}^{r} \frac{c_kb_k^{\rmH}}{s-\lambda_k},
$ and determine new interpolation points and tangential directions as described in \cite{morTomBG18}.\label{StepRed}
\EndWhile
\State Determine reduced matrices
\[
\bD_{\rr}:=  \bW^{\T}\widetilde{\bD}\bW, \quad
\bK_{\rr}:=  \bW^{\T}\bOmega^2\bW,  \quad
\bB_{\rr}:= \bW^{\T}\widetilde{\bB}, \quad
\bC_{\rr}:=  \widetilde{\bC}\bW.
\]
\end{algorithmic}
\end{algorithm}

\section{Decomposition of the controllability space }\label{sec:SatSpRep}
We consider the system response in \eqref{eq:SysResTr} and insert the definition of $\bcC$ to obtain
\[
\bJ(c,g) = \sqrt{\trace{\widetilde{\bC}\bP_{\SO}\widetilde{\bC}^{\T}}},
\]
where $\bP_{\SO}$ is the position controllability Gramian from \eqref{eq:SO_Gram} that spans the controllability space of the second-order system \eqref{eq:SOsysPhi}.
Our aim in this article is to find a basis $\bV_{\rr}$ that spans an approximation of this controllability space for all possible parameters $c\in\bcD_c$ and  $g\in\bcD_g$.
The space that contains the controllability space of system \eqref{eq:SOsysPhi} for a parameter $(c,g)\in\bcD$ is denoted by $\bcV(c,g)$ and the space that contains the controllability space for all parameters, that we aim to approximate, is denoted by $\bcV$ in the following.
For a given damping $(c,g)\in\bcD$, the controllability space $\bcV(c,g)$ of the system is spanned by
\begin{align}\label{eq:Vcg}
\bcV(c,g)
=\myspan{(s_1^2\bI_{n} + s_1\widetilde{\bD}(c,g) + \bOmega^2)^{-1}\widetilde{\bB}b_1, \dots,(s_N^2\bI + s_N\widetilde{\bD}(c,g) + \bOmega^2)^{-1}\widetilde{\bB}b_N}
\end{align}
if the parameters $s_1, \dots, s_N$ and tangential directions $b_1, \dots, b_N$ are chosen well and $N$ is chosen large enough.

%We consider a basis $\bV_{\rr}(c,g)$ that spans an approximation of the controllability space, which is equal to the solution space of the Lyapunov equation in \eqref{eq:Lyap} for parameters $c\in\bcD_c$ and $g\in\bcD_g$.

In this section, we introduce a decomposition of the controllability space into a parameter-independent space and a parameter-dependent one which can also be applied to the observability space.
Using this decomposition, we can simplify the computations within the reduction process and define an error indicator in a following section.
% or $s_1, \dots, s_N$ are chosen so, that they satisfy the interpolation conditions \eqref{eq:interpol_cond} with $N=n$.
The following theorem provides a decomposition of the controllability space $\bcV(c,g)$ into a parameter-independent subspace and several parameter-dependent ones that are used in the rest of this work.
\begin{theorem}\label{theo:decV}
Consider the asymptotically stable second-order system \eqref{eq:SOsysPhi} for a fixed parameter pair $(c,g)\in\bcD$.
Let $\bcV(c,g)$ be the controllability space of this system that satisfies \eqref{eq:Vcg} for a given set of interpolation points $s_1, \dots, s_N$ and directions $b_1, \dots, b_N$.
Define $\bLambda(s):=(s^2\bI_n + s 2\alpha\bOmega + \bOmega^2)^{-1}$, then the space $\bcV(c,g)$ satisfies
\begin{align}\label{eq:DecContrSpace}
\bcV(c,g)\subseteq \bcV_0 \cup \bcV_{\bH}(c, g) \subseteq \bcV_0 \cup \bcV_{\bF}(c)
\end{align}
for subspaces
\begin{align}\label{eq:Vspace_def}
\begin{split}
\bcV_0 :={}& \myspan{
\bLambda(s_1)\widetilde{\bB}b_1, \dots ,\bLambda(s_N)\widetilde{\bB}b_N
}, \\
\bcV_{\bH}(c, g) :={}& \myspan{
\bLambda(s_1)\widetilde{\bF}(c)\bH(s_1, b_1; c,g) , \dots ,\bLambda(s_N)\widetilde{\bF}(c)\bH(s_N, b_N; c,g)
},\\
\bcV_{\bF}(c) :={}&\myspan{
\bLambda(s_1)\widetilde{\bF}(c) , \dots ,\bLambda(s_N)\widetilde{\bF}(c)
},
\end{split}
\end{align}
and $\;\bH(s_i, b_i; c,g) := \left( \frac{1}{s_i}\bG(g)^{-1}+\widetilde{\bF}(c)^{\T}\bLambda(s_i)\widetilde{\bF}(c)\right)^{-1}\widetilde{\bF}(c)^{\T}\bLambda(s_i)\widetilde{\bB}b_i\;$
for $i=1, \dots, N$.
\end{theorem}

\begin{proof}
We consider one element corresponding to the interpolation point $s_i$ and the tangential direction $b_i$ of the span in \eqref{eq:Vcg}.
One can apply Sherman-Morrison-Woodbury formula to show that $\left( \frac{1}{s_i}\bG(g)^{-1}+\widetilde{\bF}(c)^{\T}\bLambda(s_i)\widetilde{\bF}(c)\right)^{-1}$ exists.
Now, we apply the Sherman-Morrison-Woodbury formula and obtain
\begin{align*}
(s_i^2\bI_n + s_i\widetilde{\bD}(c, g) + &\bOmega^2)^{-1}\widetilde{\bB}b_i\\
&= \left(s_i^2\bI_n + 2s_i \alpha\bOmega + \bOmega^2 + s_i \widetilde{\bF}(c)\bG(g)\widetilde{\bF}(c)^{\T}\right)^{-1}\widetilde{\bB}b_i\\
&= \left(\bLambda^{-1}(s_i) + s_i \widetilde{\bF}(c)\bG(g)\widetilde{\bF}(c)^{\T}\right)^{-1}\widetilde{\bB}b_i\\
&= \bLambda(s_i)\widetilde{\bB}b_i - \bLambda(s_i)\widetilde{\bF}(c)\left( \frac{1}{s_i}\bG(g)^{-1}+\widetilde{\bF}(c)^{\T}\bLambda(s_i)\widetilde{\bF}(c)\right)^{-1}\widetilde{\bF}(c)^{\T}\bLambda(s_i)\widetilde{\bB}b_i\\
&= \bLambda(s_i)\widetilde{\bB}b_i - \bLambda(s_i)\widetilde{\bF}(c)\bH(s_i, b_i; c,g).
\end{align*}
We observe that
\begin{multline*}%\label{eq:spaceVi}
(s_i^2\bI_n + s_i\widetilde{\bD}(c,~g) + \bOmega^2)^{-1}\widetilde{\bB}b_i \\\in \myspan{\bLambda(s_i)\widetilde{\bB}b_i,\; \bLambda(s_i)\widetilde{\bF}(c)\bH(s_i,b_i; c,g)}
\subseteq\myspan{\bLambda(s_i)\widetilde{\bB}b_i,\; \bLambda(s_i)\widetilde{\bF}(c)}.
\end{multline*}
Applying the same argument to all elements corresponding to the interpolation points $s_i$ and the tangential directions $b_i$, $i=1,\dots,N$, leads to the relation in \eqref{eq:DecContrSpace} with spaces defined as in  \eqref{eq:Vspace_def}.
\end{proof}
A direct consequence from that theorem is the following corollary which provides certain spaces that contain the controllability space for all admissible parameters.
\begin{corollary}\label{col:V}
Consider the asymptotically stable second-order system \eqref{eq:SOsysPhi}, the corresponding controllability spaces $\bcV(c,g)$, and the spaces $\bcV_0$, $\bcV_{\bH}(c,g)$, $\bcV_{\bF}(c)$ as defined in \eqref{eq:Vspace_def} for the parameter pairs $(c,g)\in\bcD$.
The spaces $\bcV$, $\bcV_{\bF}$, and $\bcV_{\bH}$ defined as
\begin{align}\label{eq:decV}
\bcV:=\bigcup_{(c,g)\in\bcD} \bcV(c,g),\qquad \bcV_{\bH}:= \bcV_0 \cup \bigcup_{(c,g)\in\bcD} \bcV_{\bH}(c,g),\qquad \bcV_{\bF} := \bcV_0 \cup \bigcup_{c\in\bcD_c}\bcV_{\bF}(c)
\end{align}
contain the corresponding controllability space for all parameters $(c,g)\in\bcD$ and they satisfy
\[
\bcV \subseteq \bcV_{\bH} \subseteq \bcV_{\bF}.
\]
\end{corollary}
Note that the space $\bcV_{\bF}(c)$ is even independent of the damping gains and depends only on the positions $c\in\bcD_{c}$ of the external dampers.
Therefore, to determine $\bcV_{\bF}$, only the spaces $\bcV_{\bF}(c)$ need to be calculated for a finite number of parameters, while for $\bcV$ and $\bcV_{\bH}$ an infinite number of parameters from $\bcD$ need to be considered to form the corresponding spaces.
However, forming the controllability spaces for all parameters in $\bcD_{c}$ is not practical, so that in the following we use the reduced basis method to approximate the different spaces from \eqref{eq:decV}.
In the following, the decomposition of the controllability spaces into a parameter-independent and into parameter-dependent parts are used to derive a reduced basis in order to reduce the system as described in \eqref{eq:SOsys_red} and to derive error indicators that quantify the quality of the reduced system.

\section{Reduction}\label{sec:Redu}
In this section, we generate a basis $\bV_{\rr}$ in an iterative scheme that is used to define a reduced optimization problem.
For the reduced second-order matrices as defined in \eqref{eq:redMatWT} with $\bW=\bV_{\rr}$ and $\bT=\bV_{\rr}$, the reduced first-order system matrices are
\[
\bcA_{\rr}(c, g):= \begin{bmatrix}
0 & \bI_r \\ -\bK_{\rr} &-\bD_{\rr}(c, g)
\end{bmatrix},\qquad \bcB_{\rr}:= \begin{bmatrix}
0\\ \bB_{\rr}
\end{bmatrix},\qquad \bcC_{\rr}:= \begin{bmatrix}
\bC_{\rr} & 0
\end{bmatrix}.
\]
Using these reduced first-order matrices, we define a reduced system response that we aim to minimize
\begin{align}\label{eq:redJ}
\bJ_{\rr}(c, g) := \sqrt{\trace{\bcC_{\rr}\bP_{\rr}(c,g)\bcC_{\rr}^{\T}}},
\end{align}
where $\bP_{\rr}(c,g)$ is the reduced controllability Gramian that solves the reduced Lyapunov equation
\begin{align}\label{eq:redJL}
\bcA_{\rr}(c, g)\bP_{\rr}(c,g)+\bP_{\rr}(c,g)\bcA_{\rr}(c, g)^{\T} = -\bcB_{\rr}\bcB_{\rr}^{\T}.
\end{align}

%\todo{This section has a bit confusing notation, we have $\bcV(c), \bcV_0, V(c), V(c,0), V(0)$ and it's not really easy to follow what is what. I think V should always depend on same amount of variables.}
We aim to construct a space $\bcV$, $\bcV_{\bH}$, or $\bcV_{\bF}$ which spans the controllability space of the system in \eqref{eq:SOsysPhi} for all admissible dampers.
For that we utilize Corollary \ref{col:V}.
However, it is numerically unfeasible to compute the spaces $\bcV(c,g)$, $\bcV_{\bH}(c,g)$, or $\bcV_{\bF}(c)$ from \eqref{eq:Vspace_def} for all $(c, g)\in\bcD$.
Hence, we aim to apply an adaptive reduced basis method scheme based on the methods presented in \cite{morSonS17,morPrzPB23} to find a basis $\bV_{\rr}$ (and the corresponding space $\bcV_{\rr}=\myspan{\bV_{\rr}}$) that is parameter-independent and that spans an approximation of the controllability space $\bcV_{\rr}\approx \bcV$.
The idea of our method is to successively add basis vectors to the basis $\bV_{\rr}$ until the approximation of the controllability space is sufficiently good for all parameters $(c, g)\in\bcD$.

First, we define the basis $\bV_0$ which spans the controllability space of the externally undamped system, i.e. $\myspan{\bV_0} =\bcV_0$ and initialize $\bV_{\rr}$ by setting $\bV_{\rr}=\bV_0$.
As we observe in \eqref{eq:DecContrSpace}, the space $\bcV_{\rr} = \bcV_0$ is included in the controllability space for all different damper's positions and, hence, is a reasonable first basis.
Then we check, whether $\bV_{\rr}$ spans a space $\bcV_{\rr}$ that is already a good approximation of the controllability space $\bcV(c,g)$ for all $(c, g)\in\bcD$, and hence, of $\bcV$.
If this is not the case, we select parameters $(c, g)\in\bcD$ for which the controllability space is not well-approximated by the current $\bcV_{\rr}$, and add the basis vectors corresponding to the space $\bcV_{\bH}(c,g)$ or $\bcV_{\bF}(c)$ into the basis $\bV_{\rr}$.
We continue with this process until the controllability spaces for all $(c, g)\in\bcD$ are approximated sufficiently good by the basis matrix $\bV_{\rr}$.

We embed this basis building into the optimization process.
For that, we define an initial reduced optimization problem \eqref{eq:redJ} with the first basis $\bV_{\rr} = \bV_0$ and run an optimization method with initial values $(c_0, g_0)$ to obtain a minimizer $(c^*, g^*)$.
Corresponding to the variables $(c^*, g^*)$ we determine a basis $\bV_{\bH}(c^*,g^*)$ or $\bV_{\bF}(c^*)$ that spans either $\bcV_{\bH}(c^*,g^*)$ or $\bcV_{\bF}(c^*)$, respectively, and that is used to enrich the basis $\bV_{\rr}$ as
\begin{align}\label{eq:Vr_enrich}
\bV_{\rr}=\mathrm{orth}([\bV_{\rr},~\bV_{\bH}(c^*, g^*)])\qquad\text{ or }\qquad\bV_{\rr}=\mathrm{orth}([\bV_{\rr},~\bV_{\bF}(c^*)]).
\end{align}
We set the new initial values $(c_0, g_0) = (c^*, g^*)$ and run a new optimization method.
We continue with this procedure until two consecutive minimizers are equal or the relative norm of their difference is smaller than a tolerance $\tol_{\mathrm{err}_1}$, which leads to Algorithm \ref{algo:RBM_Opti}.
%We observe that the parameter set $\bcD$ is only used to determine the initial values $(c_0,g_0)$ which is an advantage compared to other methods that use the parameter set to build basis $\bV_{\rr}$.
\begin{algorithm}[tb]
	\caption{Optimization using the reduced basis method}
	\label{algo:RBM_Opti}
	\begin{algorithmic}[1]
		\Require{The original system \eqref{eq:SOsysPhi}, parameter set $\bcD$, a tolerance $\tol_{\mathrm{err}_1}$.}
		\Ensure{An suboptimal damping $(c^*,g^*)$.}
		\State Set initial values $(c_0,g_0)\in\bcD$.
		\State Determine $\bV_0$.\label{step:V1}
		\State Set $\bV_{\rr} = \mathrm{orth}(\bV_0)$.
		\State Define the by $\bV_{\rr}$ reduced optimization problem $\bJ_{\rr}(c, g)$.
		\State Optimize $\bJ_{\rr}(c, g)$ as defined in \eqref{eq:redJ} to obtain the suboptimal $(c^*,g^*)$.
		\While{$\frac{\|c^*-c_0\|}{\|c^*\|}+\frac{\|g^*-g_0\|}{\|g^*\|}>\tol_{\mathrm{err}_1}$}
		\State Set $c_0=c^*$, $g_0=g^*$.
		\State Determine $\bV_{\bH}(c^*, g^*)$ or $\bV_{\bF}(c^*)$.\label{step:V2}
		\State Set $\bV_{\rr}=\mathrm{orth}([\bV_{\rr},~\bV_{\bH}(c^*, g^*)])$ or $\bV_{\rr}=\mathrm{orth}([\bV_{\rr},~\bV_{\bF}(c^*)])$.
		\State Define the by $\bV_{\rr}$ reduced optimization problem $\bJ_{\rr}(c^*, g^*)$.
		\State Optimize $\bJ_{\rr}(c, g)$ as defined in \eqref{eq:redJ} to obtain the suboptimal $(c^*,g^*)$.
		\EndWhile
	\end{algorithmic}
\end{algorithm}
Later in this work, we will introduce an error indicator that provides a criterion to stop the optimization before an optimum has been determined if the basis is not approximating the controllability space good enough for the parameters evaluated in the optimization process.
Since we approximate the solution and lack information on existence of a global minimum, we can only assert that hte obtained solution is suboptimal.

It remains to find methods that determine the bases $\bV_0$ and $\bV_{\bH}(c,g)$ or $\bV_{\bF}(c)$.
For that, we consider two different techniques.
The first one is introduced in Subsection \ref{ssec:BT_R} and uses second-order Gramians to determine bases of $\bV_{\bF}(c)$.
In the second one, described in Subsection \ref{ssec:IRKA_R}, we apply the IRKA method for second-order systems in order to approximate the bases $\bV_{\bH}(c,g)$.

\subsection{Gramian based approach}\label{ssec:BT_R}
As described in \eqref{eq:decV}, the space $\bcV_{\bF}\supseteq\bcV$ can be decomposed into a parameter-independent space $\bcV_0$ and into the parameter-dependent spaces $\bcV_{\bF}(c)$ as
\begin{align}\label{eq:spaceVi}
\bcV \subseteq \bcV_{\bF} =  \bcV_0 \cup \bigcup_{c\in\bcD_c} \bcV_{\bF}(c)
\qquad\text{
where
}\qquad
\bcV_{\bF}(c)=\myspan{\bLambda(s_1)\widetilde{\bF}(c), \dots ,\bLambda(s_N)\widetilde{\bF}(c)
}.
\end{align}
We emphasize that $\bcV_{\bF}(c)$ does not depend on the viscosities $g\in\bcD_g$.
As described in \eqref{eq:Vr_enrich}, we build an approximating controllability space $\bcV_{\rr}=\myspan{\bV_{\rr}}$ by adding the columns of the bases $\bV_{0}$ and $\bV_{\bF}(c)$ with $\bcV_{0}=\myspan{\bV_{0}}$ and  $\bcV_{\bF}(c)=\myspan{\bV_{\bF}(c)}$ into the basis $\bV_{\rr}$.

In this section, we use second-order controllability Gramians to build bases $\bV_0$ and $\bV_{\bF}(c)$ for the different parameters $c\in\bcD_{c}$.
The spaces $\bcV_0$ and $\bcV_{\bF}(c)$ span the controllability spaces of the second-order systems
\begin{align}\label{eq:homoSOsystems}
\ddot{\bx}_0(t) + 2\alpha\bOmega \dot{\bx}_0(t) + \bOmega^2\bx_0(t) = \widetilde{\bB}\bu(t),\qquad
\ddot{\bx}_c(t) + 2\alpha\bOmega \dot{\bx}_c(t) + \bOmega^2\bx_c(t) = \widetilde{\bF}(c)\bu(t).
\end{align}
According to Definition \ref{def:SO_mappingsGramians}, the controllability spaces of $\bx_0(t)$ and $\bx_c(t)$ are also described by the second-order controllability Gramians
\begin{align*}
\bP_0 := \int_{-\infty}^{\infty} \bLambda(i\omega)\widetilde{\bB}\widetilde{\bB}^{\T}\bLambda(i\omega)^{\rmH}\dd \omega, \qquad
\bP_{\bF}(c) := \int_{-\infty}^{\infty} \bLambda(i\omega)\widetilde{\bF}(c)\widetilde{\bF}(c)^{\T}\bLambda(i\omega)^{\rmH}\dd \omega
\end{align*}
for $\bLambda(s):=(s^2\bI_n + s 2\alpha\bOmega + \bOmega^2)^{-1}$.
We define the factorizations $\bP_0 = \bR_0\bR_0^{\T}$ and $\bP_{\bF}(c) = \bR_{\bF}(c)\bR_{\bF}(c)^{\T}$, that exist since the Gramians are symmetric and positive semi-definite, and set the bases to be
\begin{align}\label{eq:basisV0Vc}
\bV_0 = \mathrm{orth}(\bR_0), \qquad \bV_{\bF}(c) = \mathrm{orth}(\bR_{\bF}(c)).
\end{align}
According to Proposition \ref{prop:SO_Gram}, the second-order Gramians $\bP_0$ and $\bP_{\bF}(c)$ are the upper-left blocks of the first-order Gramians $\bcX_0$ and $\bcX_{\bF}(c)$ that are computed by solving the first-order Lyapunov equations
\begin{align}\label{eq:homoFOLE}
\bcA_0 \bcX_0 + \bcX_0\bcA_0^{\T} = -\bcB\bcB^{\T}, \qquad
\bcA_0 \bcX_{\bF}(c) + \bcX_{\bF}(c)\bcA_0^{\T} = -\bcF(c)\bcF(c)^{\T}
\end{align}
with
\begin{align}\label{eq:A0_B0}
\bcA_0:=\begin{bmatrix}
0 & \bI\\ -\bOmega^2 & -2\alpha\bOmega
\end{bmatrix},\qquad  \bcF(c):=\begin{bmatrix}
0\\ \widetilde{\bF}(c)
\end{bmatrix}, \qquad\text{ and }\qquad \bcB=\begin{bmatrix}
0\\ \widetilde{\bB}
\end{bmatrix}.
\end{align}
We observe that the four blocks within the matrix $\bcA_0$ are diagonal matrices which makes the Lyapunov solvings numerically more feasible than solving the Lyapunov equations in \eqref{eq:Lyap}.
Applying some iterative method to compute $\bcX_0$ and $\bcX_{\bF}(c)$ and their factorizations provides for low-rank factors $\bR_0$ and $\bR_{\bF}(c)$ that span approximately the same spaces as the desired $\bV_0$ and $\bV_{\bF}(c)$.
Finally, we add the computation of the bases $\bV_{0}$ and $\bV_{\bF}(c)$ into Step \ref{step:V1} and \ref{step:V2} of the Algorithm \ref{algo:RBM_Opti}.
%Later, we will derive an error estimator to evaluate the quality of the basis $\bV_{\rr}$.

\subsection{IRKA based approach}\label{ssec:IRKA_R}
According to \eqref{eq:decV}, we decompose the controllability space $\bcV_{\bH}\supseteq\bcV$ into a parameter-independent space and several parameter-dependent ones as
\begin{multline*}
\bcV \subseteq \bcV_{\bH} = \bcV_0 \cup \bigcup_{c\in\bcD_c\atop g\in\bcD_g} \bcV_{\bH}(c,g),\\
\bcV_{\bH}(c, g)=\myspan{
\bLambda(s_1)\widetilde{\bF}(c)\bH(s_1, b_1; c,g) , \dots ,\bLambda(s_N)\widetilde{\bF}(c)\bH(s_N, b_N; c,g)
}.
\end{multline*}
The space $\bcV_{\bH}(c, g)$ depends on the damping positions $c\in\bcD_c$ and on the viscosities $g\in\bcD_g$.
Again, we define an approximating space $\bcV_{\rr}=\myspan{\bV_{\rr}}$ using the parameter-independent basis $\bcV_{\rr}=\bV_0$ and by adding iteratively the bases $\bV_{\bH}(c,g)$ with $\bcV_{\bH}(c,g) = \myspan{\bV_{\bH}(c,g)}$ until the basis $\bV_{\rr}$ is sufficiently good.
The bases $\bV_0$ and $\bV_{\bH}(c,g)$ are generated by applying the IRKA method for second-order systems that was presented in Algorithm \ref{algo:IRKA_SO}.
Consequently, in Steps \ref{step:V1} and \ref{step:V2} of the Algorithm \ref{algo:RBM_Opti} we build $\bV_{\bH}(c,g)$ by setting
\[
\bV_{\bH}(c,g):=\orth\begin{bmatrix}
\bLambda(s_1)\widetilde{\bF}(c)\bH(s_1, b_1; c,g) & \dots & \bLambda(s_r)\widetilde{\bF}(c)\bH(s_r, b_r; c,g)
\end{bmatrix}
\]
with $\bH(s_i, b_i; c,g) := \left( \frac{1}{s_i}\bG(g)^{-1}+\widetilde{\bF}(c)^{\T}\bLambda(s_i)\widetilde{\bF}(c)\right)^{-1}\widetilde{\bF}(c)^{\T}\bLambda(s_i)\widetilde{\bB}b_i$ for interpolation points $s_i$ and tangential directions $b_i$, $i=1,\dots, r$, that are derived using the IRKA method presented in Algorithm \ref{algo:IRKA_SO}.
Note, that we only consider $r$ interpolation points and directions, since we aim to find a reduced approximation of the corresponding controllability space of dimension $r$.

\section{Error indicator}\label{sec:ErrorDet}
%In this section, we aim to determine an error indictor that serves as a criterion to evaluate the quality of the approximation of the controllability space $\bcV(c,g)$ by the basis $\bV_{\rr}$ for a certain parameter $(c,g)\in\bcD$.
%We assume in this section, that the basis $\bV_{\rr}$ is generated using the methods presented in Algorithm \ref{algo:RBM_Opti}.
During the optimization and basis building process presented in Algorithm \ref{algo:RBM_Opti}, two consecutive optima $(c^*, g^*)$ generated within this algorithm can have a distance smaller than a tolerance, even if the real optimum is not reached yet.
This situation might appear, if two consecutive bases $\bV_{\rr}$ span spaces that are close to each other but do not approximate the controllability space for the used parameters well.
Also, the optimization process is executed until an optimum is reached, even if the basis leads to a bad approximation of the controllability space, and hence, the corresponding optimization is meaningless.
A solution to both the problems is provided by an error indicator that evaluates the quality of the controllability space approximation by $\bV_{\rr}$.
Hence, in this section, we aim to derive an error expression that indicates the quality of the basis $\bV_{\rr}$ as an approximation of the controllability space $\bcV(c,g)$ as defined in \eqref{eq:Vcg} for parameters $(c,g)\in\bcD$.

For that, we consider the space $\bcV_{\bF}(c)$ which corresponds to the second-order system \eqref{eq:homoSOsystems} and the corresponding position controllability Gramian $\bP_{\bF}(c)$ that is the upper left block $\bX_{11}(c)$ of the first-order Gramian
\begin{align}\label{eq:X_dec}
\bcX_{\bF}(c)=\begin{bmatrix}
\bX_{11}(c) & \bX_{12}(c) \\ \bX_{12}(c)^{\T} & \bX_{22}(c)
\end{bmatrix}
\end{align}
that solves the Lyapunov equation from \eqref{eq:homoFOLE} with the right-hand side $-\bcF(c)\bcF(c)^{\T}$.
Using the basis $\bV_{\rr}$, we determine the approximation of the  Gramian
\begin{align}\label{eq:Y_dec}
\bcY_{\bF}(c)=\begin{bmatrix}
\bY_{11}(c) & \bY_{12}(c) \\ \bY_{12}(c)^{\T} & \bY_{22}(c)
\end{bmatrix}=\begin{bmatrix}
\bV_{\rr} \bY_{11,V}(c)\bV_{\rr}^{\T} & \bV_{\rr} \bY_{12,V}(c)\bV_{\rr}^{\T} \\ \bV_{\rr} \bY_{12,V}(c)^{\T}\bV_{\rr}^{\T} & \bV_{\rr} \bY_{22,V}(c)\bV_{\rr}^{\T}
\end{bmatrix}
\end{align}
where $\bY_{11}=:\widetilde{\bP}_{\bF}(c)$ is an approximation of $\bP_{\bF}(c)$.
The reduced Gramian $\bcY_{\bV}(c)$ used in this approximation is computed by solving the reduced Lyapunov equation
\begin{align}\label{eq:Y_red}
\bcA_{0,\bV}\bcY_{\bV}(c)+\bcY_{\bV}(c)\bcA_{0,\bV}^{\T}=-\bcF_{\bV}(c)\bcF_{\bV}(c)^{\T}
\end{align}
with
\[
\bcA_{0,\bV}:= \begin{bmatrix}
0 & \bI \\ -\bV_{\rr}^{\T}\bOmega^2 \bV_{\rr} & -2\alpha \bV_{\rr}^{\T}\bOmega \bV_{\rr}
\end{bmatrix},\quad \bcF_{\bV}(c):=\begin{bmatrix}
0\\ \bV_{\rr}^{\T}\widetilde{\bF}(c)
\end{bmatrix},\quad \bcY_{\bV}(c)=\begin{bmatrix}
\bY_{11,\bV}(c) & \bY_{12,\bV}(c) \\ \bY_{12,\bV}(c)^{\T} & \bY_{22,\bV}(c)
\end{bmatrix}.
\]
Now, if $\bP_{\bF}(c)$ is well-approximated by $\widetilde{\bP}_{\bF}(c)=\bV_{\rr} \bY_{11,V}(c)\bV_{\rr}^{\T}$, then
\begin{align*}
\bcV(c,g)&\subseteq \bcV_0 \cup \bcV_{\bF}(c)
=\myspan{\bP_{\bB}}\cup \myspan{\bP_{\bF}(c)}\\
&\approx\myspan{\bP_{\bB}}\cup \myspan{\widetilde{\bP}_{\bF}(c)}
=\myspan{\bP_{\bB}}\cup \myspan{\bV_{\rr}}
=\myspan{\bV_{\rr}},
\end{align*}
and hence, the controllability space $\bcV(c,g)$ is well-approximated by the basis $\bV_{\rr}$.
Consequently, evaluating the error $\bP_{\bF}(c)-\widetilde{\bP}_{\bF}(c)$ provides a criterion that indicates the quality of the approximation of the controllability space $\bcV(c,g)$ by $\bcV_{\rr}=\mathrm{span}(\bV_{\rr})$.

Inserting the decomposition of $\bcX_{\bF}(c)$ from \eqref{eq:X_dec} and $\bcY_{\bF}(c)$ from \eqref{eq:Y_dec} into the Lyapunov equation in \eqref{eq:homoFOLE} with the right-hand side $-\bcF(c)\bcF(c)^{\T}$ leads to the sub-equations
\begin{subequations}\label{eq:I_to_III}
\begin{align}
&\bX_{12}(c)^{\T} + \bX_{12}(c) = 0, \label{eq:I}\\
&\bX_{22}(c) - \bX_{11}(c)\bOmega^2 -2\alpha \bX_{12}(c)\bOmega = 0, \label{eq:II}\\
&\bOmega^2\bX_{12}(c) + \bX_{12}(c)^{\T}\bOmega^2 + 2\alpha \bX_{22}(c)\bOmega + 2\alpha\bOmega \bX_{22}(c) +\bW(c) = 0,\label{eq:III}
\end{align}
\end{subequations}
where $\bW(c):=\bF(c)\bF(c)^{\T}$ and the corresponding errors
\begin{subequations}\label{eq:EI_to_EIII}
\begin{align}
\bE_1(c)&:=\bY_{12}(c)^{\T} + \bY_{12}(c), \label{eq:EI}\\
\bE_2(c)&:=\bY_{22}(c) - \bY_{11}(c)\bOmega^2 -2\alpha \bY_{12}(c)\bOmega,  \label{eq:EII}\\
\bE_3(c)&:=\bOmega^2\bY_{12}(c) + \bY_{12}(c)^{\T}\bOmega^2 + 2\alpha \bY_{22}(c)\bOmega + 2\alpha\bOmega \bY_{22}(c) +\bW(c)\label{eq:EIII}.
\end{align}
\end{subequations}

Using the errors $\bE_{1}(c)$, $\bE_{2}(c)$, and $\bE_{3}(c)$ defined in \eqref{eq:EI_to_EIII}, we derive a formula describing the error between $\bP_{\bF}(c)$ and $\widetilde{\bP}_{\bF}(c)$.
\begin{theorem}
Consider the asymptotically stable second-order system \eqref{eq:homoSOsystems} with an input matrix $\bF(c)\in\R^{n\times\ell}$, with the corresponding first-order Gramian $\bcX_{\bF}(c)$ as defined in \eqref{eq:X_dec}, and with a reduction basis $\bV_{\rr}\in\Rnr$ that determines the matrix $\bcY_{\bF}(c)$ as defined in \eqref{eq:Y_dec}.
Then it holds
\begin{align}\label{eq:Delta}
\bDelta(c):={}& \trace{\bX_{11}(c)-\bY_{11}(c)}\nonumber\\
={}&\trace{\bE_2(c)\bOmega^{-2}} +\frac{1}{4\alpha}\left( -\trace{\bOmega^{-3}\bE_3(c)}+\trace{\bOmega^{-1}\bE_1(c)}\right)+\alpha\trace{\bE_1(c)\bOmega^{-1}}
\end{align}
with $\bE_1(c)$, $\bE_2(c)$, and $\bE_3(c)$ as defined in \eqref{eq:EI_to_EIII}.
\end{theorem}
\begin{proof}
The equations in \eqref{eq:I} and \eqref{eq:EI} yield
\begin{align}\label{eq:Sxy}
\bX_{12}(c) = \bS_X(c), \quad \bY_{12}(c) = \bS_Y(c) + \frac{1}{2}\bE_{1}(c),\quad\text{and}\quad \bX_{12}(c)-\bY_{12}(c) = \bS_X(c) - \bS_Y(c) - \frac{1}{2}\bE_{1}(c),
\end{align}
where $\bS_X(c) = -\bS_X(c)^{\T}$ and $\bS_Y(c) = -\bS_Y(c)^{\T}$ are skew-symmetric matrices.
Now, we subtract the equation in \eqref{eq:EII} from the one in \eqref{eq:II} and multiply from the right by $\bOmega^{-2}$ to obtain
\begin{align}\label{eq:IIminEII}
(\bX_{22}(c)-\bY_{22}(c))\bOmega^{-2} - (\bX_{11}(c)-\bY_{11}(c)) -2\alpha (\bX_{12}(c)-\bY_{12}(c))\bOmega^{-1} = -\bE_2(c)\bOmega^{-2}.
\end{align}
We aim to compute the error between $\bX_{11}(c)$ and $\bY_{11}(c)$.
For that, according to the equation in \eqref{eq:IIminEII}, we need to describe $(\bX_{22}(c)-\bY_{22}(c))\bOmega^{-2}$ in more detail. We subtract the equation in \eqref{eq:EIII} from the one in \eqref{eq:III} and multiply from the left by $\bOmega^{-3}$ to obtain
\begin{multline*}
\bOmega^{-1}(\bX_{12}(c)-\bY_{12}(c))+\bOmega^{-3}(\bX_{12}(c)-\bY_{12}(c))^{\T}\bOmega^2\\
+2\alpha\bOmega^{-2}(\bX_{22}(c)-\bY_{22}(c))+2\alpha\bOmega^{-3}(\bX_{22}(c)-\bY_{22}(c))\bOmega
= -\bOmega^{-3}\bE_3(c).
\end{multline*}
Utilizing the equation from \eqref{eq:Sxy} and the trace properties yields
\begin{align}\label{eq:IIIminEIII}
2&\trace{\bOmega^{-2}(\bX_{22}(c)-\bY_{22}(c))}\nonumber\\
&=\frac{1}{2\alpha}\left( -\trace{\bOmega^{-3}\bE_3(c)}-\trace{\bOmega^{-1}(\bX_{12}(c)-\bY_{12}(c))}-\trace{\bOmega^{-3}(\bX_{12}(c)-\bY_{12}(c))^{\T}\bOmega^2\right)}\nonumber\\
&=\frac{1}{2\alpha}\left( -\trace{\bOmega^{-3}\bE_3(c)}+\trace{\bOmega^{-1}\bE_1(c)}-2\trace{\bOmega^{-1}(\bS_X(c)-\bS_Y(c))}\right)\nonumber\\
&=\frac{1}{2\alpha}\left( -\trace{\bOmega^{-3}\bE_3(c)}+\trace{\bOmega^{-1}\bE_1(c)}\right)
\end{align}
since $\trace{\bOmega^{-1}(\bS_X(c)-\bS_Y(c))}=0$ because of the skew-symmetry of $\bS_X(c)-\bS_Y(c)$.
Now, we apply the trace operator to the equation in \eqref{eq:IIminEII} and insert the equation from \eqref{eq:IIIminEIII} to obtain
\begin{align*}
&\trace{\bX_{11}(c)-\bY_{11}(c)} \\
&\qquad\qquad\qquad= \trace{\bE_2(c)\bOmega^{-2}} +\trace{(\bX_{22}(c)-\bY_{22}(c))\bOmega^{-2}}-2\alpha\trace{ (\bX_{12}(c)-\bY_{12}(c))\bOmega^{-1}}\\
&\qquad\qquad\qquad= \trace{\bE_2(c)\bOmega^{-2}} +\frac{1}{4\alpha}\left( -\trace{\bOmega^{-3}\bE_3(c)}+\trace{\bOmega^{-1}\bE_1(c)}\right)+\alpha\trace{ (\bE_1(c))\bOmega^{-1}}.%-2\alpha\trace{ (S_X-S_Y)\bOmega^{-1}}
\end{align*}
\end{proof}
With this theorem, we have an explicit formula for the error in our second-order controllability Gramian, and therefore an indicator of the quality of the approximation of the controllability space for a parameter $(c,g)\in\bcD$ by the basis $\bV_{\rr}$.
This indicator $\bDelta(c)$ is used to improve the two different versions of Algorithm \ref{algo:RBM_Opti} for the two basis enrichment methods presented in the Subsections \ref{ssec:BT_R} and \ref{ssec:IRKA_R}.

Assume that we have a given basis $\bV_{\rr}$ generated with either of the two methods presented in Subsections \ref{ssec:BT_R} and \ref{ssec:IRKA_R}.
We run the optimization process and check in every step of the current optimization process whether the controllability space corresponding to the current position $c\in\bcD_c$ is well-approximated by the basis $\bV_{\rr}$.
If this is the case, i.e. if $\bDelta(c)$ is smaller than a given tolerance, we continue the optimization process.
Otherwise, we stop the optimization process, since the optimization results might be meaningless, and enrich the basis $\bV_{\rr}$ by the basis $\bV_{\bF}(c)$ or $\bV_{\bH}(c,g)$ which are either generated using the Gramian-based method presented in Subsection \ref{ssec:BT_R} or an IRKA method from Subsection \ref{ssec:IRKA_R}.
This process is described by Algorithms \ref{algo:RBM_Delta} and \ref{algo:stop}.
In Algorithm \ref{algo:stop}, we define the function that is optimized where the error $\bDelta(c)$ is computed in every function evaluation.
If the error $\bDelta(c)$ in the current parameter $c$ is larger than a certain tolerance, the method returns the boolean variable $\mathtt{conv} = 0$ that indicates that the current basis $\bV_{\rr}$ is not sufficiently good.
Algorithm \ref{algo:RBM_Delta} describes the overall optimization and basis enrichment method, where we run a reduced optimization process in Steps \ref{S3} and \ref{S6} until we either found an optimum $(c^*,g^*)\in\bcD$ or the basis $\bV_{\rr}$ is not sufficiently good, i.e. $\mathtt{conv} = 0$.
If the basis is not sufficiently good, we enrich the basis in Step \ref{step5} by adding either  $\bV_{\bF}(c^*)$ or $\bV_{\bH}(c^*, g^*)$ into the basis and proceed with this method until an optimum is found.
\begin{algorithm}[tb]
\caption{Optimization using $\bDelta(c)$}
\label{algo:RBM_Delta}
\begin{algorithmic}[1]
\Require{The original system \eqref{eq:SOsysPhi}, parameter set $\bcD$.}
\Ensure{An suboptimal damping $(c^*,g^*)\in\bcD$.}
\State Set initial values $(c_0,g_0)\in\bcD$.
\State Generate a first basis $\bV_{\rr} = \orth(\bV_0)$.
\State Optimize function $\mathtt{fun_{stop}}$ presented in Algorithm \ref{algo:stop} with an initial value $(c_0,g_0)$ to obtain $(c^*,g^*)$ and $\mathtt{conv}$.\label{S3}
\While {$\mathtt{conv} = 0$}
\State Enrich basis as $\bV_{\rr} = \orth{[\bV_{\rr},\, \bV_{\bF}(c^*)]}$ or $\bV_{\rr} = \orth{[\bV_{\rr},\, \bV_{\bH}(c^*, g^*)]}$ using either $\bV_{\bF}(c^*)$ generated as in \eqref{eq:basisV0Vc} or $\bV_{\bF}(c^*, g^*)$ determined by applying the IRKA method as in Algorithm \ref{algo:IRKA_SO}.\label{step5}
\State Optimize function $\mathtt{fun_{stop}}$ presented in Algorithm \ref{algo:stop} with initial value $(c^*,g^*)$ to obtain new values $(c^*,g^*)$ and $\mathtt{conv}$.\label{S6}
\EndWhile
\end{algorithmic}
\end{algorithm}

\begin{algorithm}[tb]
\caption{Optimizing function $\mathtt{fun_{stop}}$}
\label{algo:stop}
\begin{algorithmic}[1]
\Require{The original system \eqref{eq:SOsysPhi}, parameters $c\in\bcD_c$, $g\in\bcD_g$, tolerance $\mathrm{tol}_{\mathrm{err}_2}$, basis $\bV_{\rr}$.}
\Ensure{function value $\bJ_{\rr}(c,g)$, boolean variable $\mathtt{conv}$.}
\State Compute error $\bDelta(c)$ as defined in \eqref{eq:Delta}.
\If {$\bDelta(c) < \mathrm{tol}_{\mathrm{err}_2}$}
\State Define reduced matrices
\begin{align*}
\bD_{\rr}(c, g):= \bV_{\rr}^{\T}\widetilde{\bD}(c, g)\bV_{\rr},\quad \bK_{\rr}:= \bV_{\rr}^{\T}\bOmega^2\bV_{\rr},\quad \bB_{\rr}:= \bV_{\rr}^{\T}\widetilde{\bB},\quad \bC_{\rr}:= \widetilde{\bC}\bV_{\rr}.
\end{align*}
\State Compute reduces system response $\bJ_{\rr}$ as in \eqref{eq:redJL}.
\State Set $\mathtt{conv} = 1$.
\Else
\State Set $\mathtt{conv} = 0$.
\State Stop optimization process.
\EndIf
\end{algorithmic}
\end{algorithm}

\section{Computational details}\label{sec:CompDet}
In this section, we describe some computational details. In Subsection \ref{ssec:Diag}, we introduce a diagonalization of the matrices in the Lyapunov equations in \eqref{eq:homoFOLE}, which accelerates the solving process.
In addition, Subsection \ref{ssec:contOpt} describes the objective function formulation for the optimization process, especially for the discrete position values.

%%%%% Diagonalization of the Lyapunov equations
\subsection{Diagonalization of the Lyapunov equations}\label{ssec:Diag}
We consider the Lyapunov equations from \eqref{eq:homoFOLE}, where the matrix $\bcA_0$ consists of four blocks of diagonal matrices.
We can easily diagonalize such a matrix $\bcA_0$ to make the solving process of the Lyapunov equations even more efficient.
We consider the first equation in \eqref{eq:homoFOLE} with the right-hand side $-\bcB\bcB^{\T}$ while the second one is treated equally.
\begin{lemma}
Consider the Lyapunov equation from \eqref{eq:homoFOLE} with the right-hand side $\bcB\bcB^{\T}$, where $\bcA_0$ is defined as in \eqref{eq:A0_B0}.
Then there exist a nonsingular matrix $\bT$ such that $\bcA_{\rmD}=\bT\bcA_0\bT^{-1}$ is a diagonal matrix and $\bcX_{\rmD}=\bT^{-1}\bcX_0\bT$, where
\begin{align}\label{eq:lyap_diag}
\bcA_{\rmD} \bcX_{\rmD} +  \bcX_{\rmD} \bcA_{\rmD}^{\T} = -\bcB_{\rmD} \bcB_{\rmD}^{\T}
\end{align}
for $\bB_{\rmD}=\bT^{-1}\bB$.
\end{lemma}
\begin{proof}
Let $\bcP \in \R^{2n \times 2n}$ be the perfect shuffle permutation, which reorders a set of even cardinality by mapping the $k$-th entry as follows
\[
k \mapsto
\begin{cases}
2k-1, & \mbox{if } k\leq n, \\
2(k-n), & \mbox{if } k>n
\end{cases}
\]
and which satisfies $\bcP^{\T}\bcP = \bI$.
Then we get that
\[
\bcP^{\T}\bcA_0\bcP = \diag{\bcD_1, \dots,\bcD_n}, \quad
\bD_j := \begin{bmatrix}
0 & \omega_j \\
-\omega_j & -2\alpha \omega_j
\end{bmatrix},\qquad j = 1,\dots, n.
\]
Now, let $\bPsi_j$ be the matrix which diagonalizes the matrix $\bD_j$, i.e. $\bPsi_j^{-1}\bD_{j}\bPsi_j$ is a diagonal matrix,  for $j = 1, \dots, n$.
We define the matrix $\bPsi := \diag{\bPsi_1 ,\dots , \bPsi_n}$, so that we obtain
\[
\bcA_{\rmD}:= \bT^{-1}\bcA_0\bT,\qquad  \bcB_{\rmD} := \bT^{-1}\bcB\qquad\text{ for } \qquad\bT:=\bcP \bPsi
\]
where $\bcA_{\rmD}$ is a diagonal matrix.
Multiplying the Lyapunov equation from \eqref{eq:homoFOLE} with the right-hand side $\bcB\bcB^{\T}$ from the left and from the right by $\bT^{-1}$ and $\bT$, respectively, leads to
\[
\bT^{-1}\bcA_{0}\bT \bT^{-1}\bcX_{0}\bT +  \bT^{-1}\bcX_{0}\bT \bT^{-1}\bcA_{0}^{\T}\bT = -\bT^{-1}\bcB \bcB^{\T}\bT,
\]
and hence, the Lyapunov equation \eqref{eq:lyap_diag} is solved by $\bcX_{\rmD}=\bT^{-1}\bcX_{0}\bT$.
\end{proof}

%%%%%%%%%%%%%%%%%%%%%%%%%%%%%%%%%%%%%%%%%%%%%%%%%%%%%%%%%%%%%%%%%%%%%%%%%%%%%%%%
% Optimization of postions
%%%%%%%%%%%%%%%%%%%%%%%%%%%%%%%%%%%%%%%%%%%%%%%%%%%%%%%%%%%%%%%%%%%%%%%%%%%%%%%%

\subsection{Optimization}\label{ssec:contOpt}
Our aim is to optimize the positions and the gains of the external dampers by minimizing the system response as defined in \eqref{eq:SysResTr}.
However, the damper positions $c\in\bcD_c$ are discrete values, and hence, optimization algorithms such as the Nelder-Mead method are not directly applicable.
Therefore, we modify the system response to be
\begin{align}\label{eq:opti_roundPos}
\widetilde{\bJ}(c, g) := \bJ([c], g)= \bJ([c_1], \dots, [c_{\ell}], g_1,\dots, g_{\ell})
\end{align}
as described in \cite{KanPTT19}, so that the positions $c$, that are optimized, are continuous variables.
When only the damper positions are to be optimized, i.e. the damping gains $g^* = [g_1^*,\dots, g_{\ell}^*]\in\mathbb{R}^{\ell}$ are fixed, we define a second system response function, that is
\begin{align}\label{eq:opti_Pos}
\widetilde{\bJ}_{\pos}(c) := \bJ([c], g^*)= \bJ([c_1], \dots, [c_{\ell}], g_1^*,\dots, g_{\ell}^*).
\end{align}
Optimizing this function using standard optimization methods such as the Nelder-Mead method might cause convergence problems if the step size is too small because of the discrete function values of \eqref{eq:opti_Pos}.
Hence, we modify this function to make the function values continuous.
We first consider the case where we only have one damper with the position $c\in\mathbb{R}_+$.
We split the current position value $c = c^{\intt} + c^{\dec}$ where $c^{\intt} := \lfloor c \rfloor$ and $c^{\dec} = c-c^{\intt}$.
The corresponding function value is then defined as
\[
\widehat{\bJ}_{\pos}(c):= \left(1-c^{\dec}\right)\bJ\left(c^{\intt}, g^*\right) + c^{\dec}\bJ\left(c^{\intt}+1, g^*\right)
\]
which provides a linear interpolation between the function values corresponding to two discrete damper's positions.

This idea is now generalized for $\ell$ dampers, i.e. $c\in\mathbb{R}^{\ell}_+$.
We define for $c = [c_1,\dots,c_{\ell}]\in\R_+^{\ell}$ the values
\begin{align*}
c_j := c_j^{\intt} + c_j^{\dec} \quad\text{with}\quad c_j^{\intt} := \lfloor c_j \rfloor,\;\; c_j^{\dec} = c_j-c_j^{\intt},\qquad \text{ for } j = 1,\dots, \ell.
\end{align*}
Accordingly, we define the continuous function values as
\begin{align*}
\widehat{\bJ}_{\pos}(c) := &\left(1-c_1^{\dec}\right)\cdots\left(1-c_{\ell}^{\dec}\right)\bJ\left([c_1^{\intt},\dots, c_{\ell}^{\intt}], g^*\right)\\
&\quad+c_1^{\dec}\left(1-c_2^{\dec}\right)\cdots\left(1-c_{\ell}^{\dec}\right)\bJ\left([c_1^{\intt}+1,\, c_2^{\intt},\dots,c_{\ell}^{\intt}], g^*\right)\\
&\quad+\left(1-c_1^{\dec}\right)\cdots\left(1-c_{\ell-1}^{\dec}\right) c_{\ell}^{\dec}\bJ\left([c_1^{\intt},\dots, c_{\ell-1}^{\intt},c_{\ell}^{\intt}+1], g^*\right)\\
&\hspace*{100pt}\vdots\\
&\hspace*{50pt}+\left(1-c_1^{\dec}\right)c_2^{\dec}\cdots c_{\ell}^{\dec}\bJ\left([c_1^{\intt}, c_2^{\intt}+1,\dots, c_{\ell}^{\intt}+1], g^*\right)\\
&\hspace*{50pt}+c_1^{\dec}\cdots c_{\ell-1}^{\dec}\left(1-c_{\ell}^{\dec}\right)\bJ\left([c_1^{\intt}+1,\dots, c_{\ell-1}^{\intt}+1,c_{\ell}^{\intt}], g^*\right)\\
&\hspace*{150pt}+c_1^{\dec}\cdots c_{\ell}^{\dec}\bJ\left([c_1^{\intt}+1,\dots,c_{\ell}^{\intt}+1], g^*\right).
\end{align*}
We observe that the computation of the function values of $\widehat{\bJ}_{\pos}(c)$ is only accessible for a small amount of external dampers or for very small system dimensions since the number of function evaluations rises exponentially with the numbers of dampers, where we need $2^{\ell}$ Lyapunov equation solves if $\ell$ is the number of the dampers.

Note that these convergence problems can also occur for the function defined in \eqref{eq:opti_roundPos}, however in our examples that was not the case, and hence, if both the positions and the gains are optimized, we use the system response function defined in \eqref{eq:opti_roundPos}.

\section{Numerical results}\label{sec:NumRes}
In this section, we apply the methods presented above for two numerical examples.
For both examples, we first optimize only the damper's positions and afterwards extend the approach to optimize the positions and the damping viscosities simultaneously.
We emphasise again that the system response optimized in this work is a non-convex function and has several local minima.
In this work, we determine approximations of the local minima of the system response using the full system.
If global optima are to be found, different initial conditions must be applied, see \cite{morBenTT11}.

The computations have been done on a computer with 2 Intel Xeon Silver 4110 CPUs running at 2.1 GHz and equipped with 192 GB total main memory.
We use \matlab 2021a to run the examples and for the optimization, we use the Nelder–Mead method encoded in the \matlab-function $\mathtt{fminsearch}$.

\subsection{First example}

\begin{figure}
\centering
\scalebox{1.2}{%\documentclass[tikz]{standalone}
%\usepackage{tikz}
%\usetikzlibrary{calc,patterns,decorations.pathmorphing,decorations.markings}
%
%\begin{document}
%
%
%\begin{center}

% \begin{tikzpicture}[every node/.style={draw,outer sep=0pt,thick}]
 \begin{tikzpicture}[scale=1, every node/.style={scale=0.8}]
\tikzstyle{spring}=[thick,decorate,decoration={zigzag,pre
length=0.3cm,post length=0.3cm,segment length=6}]
\tikzstyle{damper}=[thick,decoration={markings,
  mark connection node=dmp,
  mark=at position 0.5 with
  {
    \node (dmp) [thick,inner sep=0pt,transform
    shape,rotate=-90,minimum width=10pt,minimum height=2pt,draw=none]
    {};
    \draw [thick] ($(dmp.north east)+(1.5pt,0)$) -- (dmp.south east)
    -- (dmp.south west) -- ($(dmp.north west)+(1.5pt,0)$);
    \draw [thick] ($(dmp.north)+(0,-3pt)$) -- ($(dmp.north)+(0,3pt)$);
  }
}, decorate]
\tikzstyle{springdot}=[thick,decoration={markings,
  mark connection node=sdt,
  mark=at position 0.5 with
  {
  \node (sdt) [thick,inner sep=0pt,transform shape,rotate=-90,minimum
  width=0.85cm,minimum height=0.50cm,draw=none] {};
      \draw [dotted, thick, color=blue] ($(sdt.north)$) --
      ($(sdt.south)$);
  }
}, decorate]

\tikzstyle{ground}=[fill,pattern=north east lines,draw=none,minimum
width=0.75cm,minimum height=0.3cm]

\newcommand{\drawLinewithBG}[2]
{
    \draw[white,myBG]  (#1) -- (#2);
    \draw[black,very thick] (#1) -- (#2);
}

\tikzstyle myBG=[line width=3pt,opacity=1.0]

\node (M) [draw,outer sep=0pt,thick,minimum width=1.2cm, minimum
height=1.2cm,color=red, fill=red!20!white] at (0,0) {$\color{black}
m_1$};
\node (M2) [draw,outer sep=0pt,thick,minimum width=1.2cm, minimum
height=1.2cm,color=red,fill=red!20!white] at (2,0) {$\color{black}
m_2$};
\node (M3) [draw,outer sep=0pt,thick,minimum width=1.2cm, minimum
height=1.2cm,color=red,fill=red!20!white] at (4,0) {$\color{black}
m_{n-1}$};
\node (M4) [draw,outer sep=0pt,thick,minimum width=1.2cm, minimum
height=1.2cm,color=red,fill=red!20!white] at (6,0) {$\color{black}
m_n$};

\node (LW)
[ground,rotate=-90,anchor=north,xshift=-1cm,yshift=-2.0cm,minimum
width=3cm, style={draw,outer sep=0pt,thick}] at (M.south) {};

\node (RW)
[ground,rotate=90,anchor=north,xshift=1cm,yshift=-2.5cm,minimum
width=3cm, style={draw,outer sep=0pt,thick}] at (M4.south west) {};

\draw [spring,color=blue] ({-1.60,0}) -- ({-0.48,0});
\draw [spring,color=blue] (M2.180) -- ($(M.north
east)!(M.180)!(M.south east)$);
\draw [springdot,color=blue] (M3.180) -- ($(M2.north
east)!(M2.180)!(M2.south east)$);
\draw [spring,color=blue] (M4.180) -- ($(M3.north
east)!(M3.180)!(M3.south east)$);
\draw [spring,color=blue] ({6.48,0}) -- ({7.52,0});
\draw [damper] ({0,1.1}) -- ({-1.1,1.1});
\draw [thick] ({0,1.1}) -- ({0,.5});

\draw [damper] ({4,1.1}) -- ({2.9,1.1});
\draw [thick] ({4,1.1}) -- ({4,.5});

\node at (-.8,-.3) {$k_1$};
\node at (-.7,.7) {$v_1$};

\node at (1,-.3) {$k_2$};

\node at (3.3,.7) {$v_2$};
\node at (5,-.3) {$k_n$};
\node at (7.15,-.3) {$k_{n+1}$};

\node (LW)
[ground,rotate=-90,anchor=north,xshift=-1cm,yshift=-2.0cm,minimum
width=1cm, style={draw,outer sep=0pt,thick}] at (.5,.3) {};
\node (LW)
[ground,rotate=-90,anchor=north,xshift=-1cm,yshift=-2.0cm,minimum
width=1cm, style={draw,outer sep=0pt,thick}] at (4.5,.3) {};

\end{tikzpicture}

%
%\end{center}
%
%\end{document}
}
\caption{Sketch of the first example including one row of masses, that are connected with consecutive springs.}
\label{fig:3diag_ex}
\end{figure}
The first example, we consider, is described by Figure \ref{fig:3diag_ex} which results in a system of the form \eqref{eq:SOsys} with the mass matrix that is defined in \matlab-notation as
\[
\bM =  \mathtt{sparse(diag([logspace(-1,1,n/2), flip(logspace(-1,1,n/2))]))}
\]
which leads to mass values between $0.1$ and $1,$ where the largest mass values are attained by the middle masses and the outermost masses have the smallest mass values.
Moreover, the stiffness matrix is given as
\begin{align*}
K = \left[ \begin{array}{rrrrr}
24 & -20 & & & \\
-20 & 40 & \ddots & &  \\
& \ddots & \ddots & \ddots  & \\
& &\ddots & 40 & -20   \\
& &  & -20 & 20
\end{array}\right].
\end{align*}
We build the internal damping matrix using the multiple of the critical damping as described in \eqref{eq:intDamp} with $\alpha=0.005$.
The dimension is set to be $n=1000$ so that Lyapunov equations of dimension $2n=2000$ need to be solved multiple times.
Additionally, the input matrix is defined as zero matrix except for the entries $\bB_{1} = 1$, $\bB_{500} = 1$, and $\bB_{1000} = 1$
so that an external force is applied to the first mass, the middle mass and the last mass.
The output matrix $\bC$ is of dimension $n\times 3$ and has zero entries everywhere besides on $\bC_{1,10}=1$, $\bC_{2, 500}=1$, and $\bC_{3, 990}=1$.

We add two grounded dampers to the system at the positions $k$ and $j$ so that $c = [k,\; j]$ and $\bF(c) = [e_k, \; e_j]$ where $e_k$ and $e_j$ are the $k$-th and the $j$-th unit vector.
We apply the optimization method $\mathtt{fminsearch}$ that is implemented in \matlab\, to optimize the external dampers.
We stop the optimization process when the relative error in the function values or the difference between two consecutive values is smaller than the tolerance $\mathrm{tol}_{\mathrm{opt}}=10^{-3}$.
Additionally, in Algorithm \ref{algo:RBM_Opti} we stop the optimization if two consecutive optimizers have a relative error that is smaller than $\mathrm{tol}_{\mathrm{err}_1}=10^{-2}$.
Moreover, in Algorithm \ref{algo:RBM_Delta} where the error indicator $\bDelta(c)$ is used, the relative error must be smaller than the tolerance $\tol_{\mathrm{err}_2}=10^{-4}$ to continue the optimization process.

First, we only optimize the positions of the two dampers and set the damping gains to be the fixed values $g_1=g_2=1000$.
The initial positions are chosen to be $c_0 = [50, 90]$.
\begin{table}[H]
\begin{center}
\renewcommand{\arraystretch}{1.5}
\begin{tabular}{ |p{1.95cm}|p{1.90cm}|p{1.90cm}|p{2.20cm}| p{1.90cm}|p{2.20cm}|}\hline
& Original & $\bcV_{\bF}$ &$\bcV_{\bF}$ with $\bDelta(c)$ & $\bcV_{\bH}$ & $\bcV_{\bH}$ with $\bDelta(c)$\\
\hline\hline
Time &	 $2.8\cdot 10^{4}$ & $2.2\cdot 10^{2}$ &  $2.1\cdot 10^{2}$ & $2.9\cdot 10^{1}$ & $1.7\cdot 10^{2}$\\\hline
Dimension & $1000$ & $164$ & $178$ & $188$ & $93$\\\hline
Position &
$500,\;990$ &
$501,\;990$ &
$500,\;991$ &
$501,\;990$ &
$500,\;991$\\\hline
Error & - & $9.0\cdot 10^{-4}$ & $9.0\cdot 10^{-4}$ & $9.0\cdot 10^{-4}$ &  $9.0\cdot 10^{-4}$\\\hline
Acceleration & - & $129$ & $138$ & $97$ & $167$\\
\hline
\end{tabular}
\end{center}
\caption{Position optimization times and errors for the original and reduced optimization processes for the first example.}
\label{tab:1000PosOnly}
\end{table}
Table \ref{tab:1000PosOnly} presents the results for this example, where the first row shows the optimization times including the building process of the reduced models and the optimization processes.
The rows below show the dimension of the final reduced system, the suboptimal positions determined by the method of choice, and the resulting relative errors in the position.
Finally, the last row contains the acceleration rates that result from the original time divided by the reduced times.
We evaluate in the first column the original system optimization.
In the second and the third column, we describe the values for the two methods that are based on the Gramian computation from Subsection \ref{ssec:BT_R} which approximate the space $\bcV_{\bF}$ without and with the usage of the error indicator $\bDelta(c)$ that is presented in Section \ref{sec:ErrorDet}.
The remaining columns show the two methods based on the IRKA method from Subsection \ref{ssec:IRKA_R} which approximate the space $\bcV_{\bH}$, where we again consider the method using the difference between two consecutive optima and the error indicator as stopping criterion.
%We observe for this example that the IRKA based method that uses the error indicator leads to an error that is significantly larger than those for the other methods.
%The corresponding optimization method leads to positions in similar areas as the original model but stops too early to reach the right optima.
We observe that all four methods lead to positions that coincide or are only differing by one or two positions which results in errors below one percent which is sufficient for these examples.
The method that approximates $\bcV_{\bH}$ with an error indicator leads to the largest acceleration rate of $167$.
However, also the Gramian based methods that approximate the space $\bcV_{\bF}$ result in acceleration rates of $129$ and $138$ which are satisfying results.
For this specific example, the error indicator leads to an acceleration in both methods.

We additionally consider the case where we optimize the damper's positions and the corresponding gains, simultaneously, which leads to the results presented in Table \ref{tab:1000}.
The initial positions are again chosen to be $c_0 = [50, 90]$ and the initial gains are $g_0 = [1000, 1000]$.
The table has the same structure as for the previous case where no gains are optimized, however, one row is added that shows the damping gains and one row that shows the corresponding relative error between the optimal gains computed using the full system and the ones computed using the reduced ones.
We observe that if the positions and the gains are optimized, the acceleration rates are smaller which means optimizing both parameters requires more computational time than expected.
However, all the methods lead to sufficiently good results.
We observe that the Gramian-based methods, that approximate the space $\bcV_{\bF}$, lead to more accurate approximations while the IRKA based one without error indicator $\bDelta(c)$, approximating $\bcV_{\bH}$, lead to least accurate but still sufficiently good results, and the one with error indicator got stuck in a different suboptimal solution.
Also, we see that the error indicator $\bDelta(c)$ leads to smaller bases, and hence, for both methods accelerates the optimization process.
To summarize, in this case we observe that the $\bcV_{\bF}$-based methods lead to faster and more precise results.

\begin{table}[H]
\begin{center}
\renewcommand{\arraystretch}{1.5}
\begin{tabular}{ |p{2.15cm}|p{1.90cm}|p{1.90cm}|p{2.15cm}| p{1.90cm}|p{2.65cm}|}\hline
& Original & $\bcV_{\bF}$ &$\bcV_{\bF}$ with $\bDelta(c)$ & $\bcV_{\bH}$ & $\bcV_{\bH}$ with $\bDelta(c)$\\
\hline\hline
Time &	 $1.6\cdot 10^{4}$ & $5.8\cdot 10^{2}$ &  $2.0\cdot 10^{2}$ & $3.7\cdot 10^{2}$ & $2.3\cdot 10^{2}$\\\hline
Dimension & $1000$ & $262$ & $178$ & $155$ & $103$\\\hline
Position & $
35,\; 395
$ & $
35,\; 395
$ & $
35,\; 395
$  & $
35,\; 395
$ & $
72,\; 308
$\\\hline
Gain &
$\begin{matrix}
1.4234\cdot 10^{3},\\ 3.3809
\end{matrix}$ &
$\begin{matrix}
1.4234\cdot 10^{3},\\ 3.3809
\end{matrix}$  &
$\begin{matrix}
1.4234\cdot 10^{3},\\ 3.3791
\end{matrix}$  &
$\begin{matrix}
1.4223\cdot 10^{3},\\ 3.3809
\end{matrix}$  &
$\begin{matrix}
2.1117\cdot 10^{1},\\ 3.0862
\end{matrix}$
\\\hline
Error position & - & $0$ & $0$ & $0$ &  $8.0\cdot 10^{-3}$\\\hline
Error gain & - & $0$ & $1.6\cdot 10^{-6}$ & $7.8\cdot 10^{-4}$ &  $9.9\cdot 10^{-1}$\\\hline
Acceleration & - & $27$ & $81$ & $42$ & $69$\\\hline
\end{tabular}
\end{center}
\caption{Position and viscosity optimization times and errors for the original and reduced optimization processes for the first example.}
\label{tab:1000}
\end{table}

\subsection{Second example}

\begin{figure}
\centering
\scalebox{1.2}{%\documentclass[tikz]{standalone}
%\usepackage{tikz}
%\usetikzlibrary{calc,patterns,decorations.pathmorphing,decorations.markings}
%
%\begin{document}

%\begin{center}
 %\begin{tikzpicture}[every node/.style={draw,outer sep=0pt,thick}]
\begin{tikzpicture}[scale=0.8, every node/.style={scale=0.65}]
\tikzstyle{spring}=[thick,decorate,decoration={zigzag,pre
length=0.3cm,post length=0.3cm,segment length=6}]
\tikzstyle{damper}=[thick,decoration={markings,
  mark connection node=dmp,
  mark=at position 0.5 with
  {
    \node (dmp) [thick,inner sep=0pt,transform
    shape,rotate=-90,minimum width=10pt,minimum height=2pt,draw=none]
    {};
    \draw [thick] ($(dmp.north east)+(1.5pt,0)$) -- (dmp.south east)
    -- (dmp.south west) -- ($(dmp.north west)+(1.5pt,0)$);
    \draw [thick] ($(dmp.north)+(0,-3pt)$) -- ($(dmp.north)+(0,3pt)$);
  }
}, decorate]
\tikzstyle{springdot}=[thick,decoration={markings,
  mark connection node=sdt,
  mark=at position 0.5 with
  {
  \node (sdt) [thick,inner sep=0pt,transform shape,rotate=-90,minimum
  width=0.85cm,minimum height=0.50cm,draw=none] {};
      \draw [dotted, thick, color=blue] ($(sdt.north)$) --
      ($(sdt.south)$);
  }
}, decorate]

\tikzstyle{ground}=[fill,pattern=north east lines,draw=none,minimum
width=0.75cm,minimum height=0.3cm]

\newcommand{\drawLinewithBG}[2]
{
    \draw[white,myBG]  (#1) -- (#2);
    \draw[black,very thick] (#1) -- (#2);
}

\tikzstyle myBG=[line width=3pt,opacity=1.0]

\node (M) [draw,outer sep=0pt,thick,minimum width=1.4cm, minimum
height=1.2cm,color=red, fill=red!20!white] at (-0.4,0) {$\color{black}
m_{d+1}$};
\node (M2) [draw,outer sep=0pt,thick,minimum width=1.4cm, minimum
height=1.2cm,color=red,fill=red!20!white] at (1.8,0) {$\color{black}
m_{d+2}$};
\node (M3) [draw,outer sep=0pt,thick,minimum width=1.4cm, minimum
height=1.2cm,color=red,fill=red!20!white] at (4,0) {$\color{black}
m_{2d}$};
\node (M4) [draw,outer sep=0pt,thick,minimum width=1.4cm, minimum
height=6.2cm,color=red,fill=red!20!white] at (6.2,0) {$\color{black}
m_{3d+1}$};

\node (LW)
[ground,rotate=-90,anchor=north,xshift=0cm,yshift=-2.2cm,minimum
width=7cm, style={draw,outer sep=0pt,thick}] at (M) {};

\node (RW)
[ground,rotate=90,anchor=north,xshift=0cm,yshift=-2.8cm,minimum
width=7cm, style={draw,outer sep=0pt,thick}] at (M4.west) {};

\draw [spring,color=blue] ({-2.15,0}) -- ({-0.95,0});
\draw [spring,color=blue] (M2.180) -- ($(M.north
east)!(M.180)!(M.south east)$);
\draw [springdot,color=blue] (M3.180) -- ($(M2.north
east)!(M2.180)!(M2.south east)$);
\draw [spring,color=blue] (M4.180) -- ($(M3.north
east)!(M3.180)!(M3.south east)$);
\draw [spring,color=blue] ({6.76,0}) -- ({7.88,0});

\draw [damper] ({4,1.1}) -- ({2.9,1.1});
\draw [thick] ({4,1.1}) -- ({4,.5});

\node at (-1.5,-.3) {$k_2$};
\node at (0.8,-.3) {$k_2$};
\node at (5.1,-.3) {$k_2$};
\node at (7.35,-.3) {$k_{4}$};

\node at (3.4,.7) {$v$};

\node (LW)
[ground,rotate=-90,anchor=north,xshift=-1cm,yshift=-2.0cm,minimum
width=1cm, style={draw,outer sep=0pt,thick}] at (4.5,.3) {};
%============================ first row
\node (MM) [draw,outer sep=0pt,thick,minimum width=1.4cm, minimum
height=1.2cm,color=red, fill=red!20!white] at (-0.4,2) {$\color{black}
m_{1}$};
\node (MM2) [draw,outer sep=0pt,thick,minimum width=1.4cm, minimum
height=1.2cm,color=red,fill=red!20!white] at (1.8,2) {$\color{black}
m_{2}$};
\node (MM3) [draw,outer sep=0pt,thick,minimum width=1.4cm, minimum
height=1.2cm,color=red,fill=red!20!white] at (4,2) {$\color{black}
m_{d}$};
\node (MM4) at (5.75,2){};
\draw [spring,color=blue] ({-2.15,2}) -- ({-0.95,2});
\draw [spring,color=blue] (MM2.180) -- ($(MM.north
east)!(MM.180)!(MM.south east)$);
\draw [springdot,color=blue] (MM3.180) -- ($(MM2.north
east)!(MM2.180)!(MM2.south east)$);
\draw [spring,color=blue] (MM4.180) -- ($(MM3.north
east)!(MM3.180)!(MM3.south east)$);
\node at (-1.5,1.7) {$k_1$};
\node at (0.8,1.7) {$k_1$};
\node at (5.1,1.7) {$k_1$};
\node (LW)
[ground,rotate=-90,anchor=north,xshift=-1cm,yshift=-2.0cm,minimum
width=1cm, style={draw,outer sep=0pt,thick}] at (0.1,2.3) {};

\draw [damper] ({-0.4,3.1}) -- ({-1.5,3.1});
\draw [thick] ({-0.4,3.1}) -- ({-0.4,2.5});

\node at (-1.0,2.7) {$v$};
\node at (1.2,-1.3) {$v$};
%============================ third row
\node (MMM) [draw,outer sep=0pt,thick,minimum width=1.4cm, minimum
height=1.2cm,color=red, fill=red!20!white] at (-0.4,-2) {$\color{black}
m_{2d+1}$};
\node (MMM2) [draw,outer sep=0pt,thick,minimum width=1.4cm, minimum
height=1.2cm,color=red,fill=red!20!white] at (1.8,-2) {$\color{black}
m_{2d+2}$};
\node (MMM3) [draw,outer sep=0pt,thick,minimum width=1.4cm, minimum
height=1.2cm,color=red,fill=red!20!white] at (4,-2) {$\color{black}
m_{3d}$};
\node (MMM4) at (5.75,-2){};
\draw [spring,color=blue] ({-2.15,-2}) -- ({-0.95,-2});
\draw [spring,color=blue] (MMM2.180) -- ($(MMM.north
east)!(MMM.180)!(MMM.south east)$);
\draw [springdot,color=blue] (MMM3.180) -- ($(MMM2.north
east)!(MMM2.180)!(MMM2.south east)$);
\draw [spring,color=blue] (MMM4.180) -- ($(MMM3.north
east)!(MMM3.180)!(MMM3.south east)$);
\node (LW)
[ground,rotate=-90,anchor=north,xshift=-1cm,yshift=-2.0cm,minimum
width=1cm, style={draw,outer sep=0pt,thick}] at (2.3,-1.7) {};

\draw [damper] ({1.8,-0.9}) -- ({.7,-0.9});
\draw [thick] ({1.8,-0.9}) -- ({1.8,-1.5});

\node at (-1.5,-2.3) {$k_3$};
\node at (0.8,-2.3) {$k_3$};
\node at (5.1,-2.3) {$k_3$};

\end{tikzpicture}

%\end{center}

%\end{document}
}
\caption{Sketch of the first example including three rows of masses, that are connected with consecutive masses.}
\label{fig:3row_ex}
\end{figure}
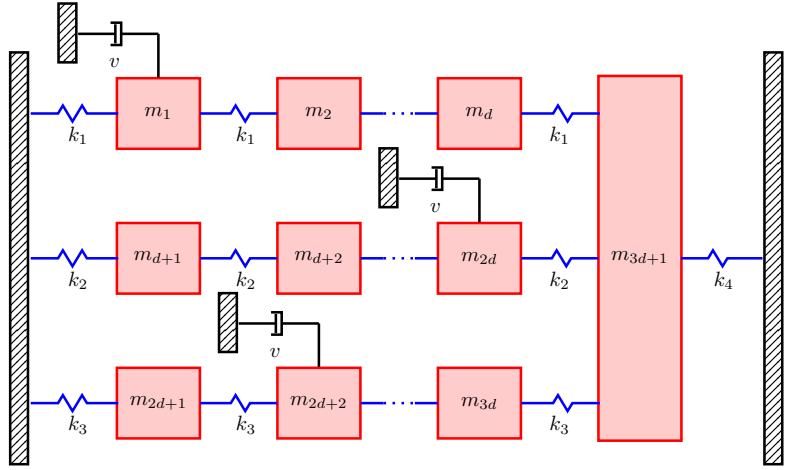
In this subsection, we consider the three row system that is depict in Figure \ref{fig:3row_ex}.
The masses are given by the \matlab-expression
\[
\bM = \mathtt{sparse(diag([logspace(3,5,ceil(n/2)), flip(logspace(3,5,floor(n/2)))]));}
\]
so that we consider mass values between $10^{3}$ and $10^{5}$, with the largest values being attained by the middle masses and the smallest values by the outermost masses.
The stiffness matrix is built as
\begin{align*}
\bK = \begin{bmatrix}
K_{11} & & & \kappa_1 \\
& K_{22} &  & \kappa_2 \\
&  & K_{33} & \kappa_3 \\
\kappa_1^\T & \kappa_2^\T & \kappa_3^{\T} & k_1+k_2+k_3+k_4
\end{bmatrix},\quad K_{ii} = k_i \begin{bmatrix}
2  & -1     &        & & \\
-1 & 2      & -1     & &  \\
& \ddots & \ddots & \ddots  &\\
&        & -1     &2        & -1 \\
&        &        &-1 & 2
\end{bmatrix},\quad \kappa_i = \begin{bmatrix}
0\\
\vdots \\
0\\
k_i
\end{bmatrix},
\end{align*}
with $k_1=20$, $k_2=10$, $k_3=5$, $k_4=20$.
The dimension of $\bM$ and $\bK$ is $n=901$, so that we have to solve Lyapunov equations of dimension $2n=1802$.
The input matrix is chosen to be $\bB = \mathtt{ones(n,1)}$, so that the external forces act uniformly on all masses, and the output matrix $\bC$ is the zero matrix of dimension $3\times n$ with non-zeros entries
$\bC_{1,10} = 1$, $\bC_{2,450} = 1$, and $\bC_{3,891} = 1$.
We assume that there are three grounded dampers  so that $\bF = [e_i,~e_j,~e_k]$ for $i,~j,~k\in\{1, \dots, n\}$.

Again, we initially consider the case where only the damper's positions are optimized, where we stop the optimization process when the relative error in the function values or the difference between two consecutive values is smaller than the tolerance $\mathrm{tol}_{\mathrm{opt}}=10^{-3}$.
The two methods without the error indicator described in Algorithm \ref{algo:RBM_Opti} are interrupted if two consecutive minimizers have a relative error that is smaller than $\mathrm{tol}_{\mathrm{err}_1}=10^{-2}$.
Moreover, in Algorithms \ref{algo:RBM_Delta} where the error expression $\bDelta(c)$ is used, the error must be smaller than the tolerance $\tol_{\mathrm{err}_2}=10^{-1}$ to continue the optimization process.
The corresponding results of the position optimization are given in Table \ref{tab:Ex3_PosOnly} where we chose the initial positions $c_0 = [100,~300,~500]$.
We observe that the positions obtained by optimizing the full problem are still approximated good enough in the sense that the error is smaller than $5\%$ for all the methods. The best approximation is obtained by applying the $\bcV_{\bF}$-based methods while the $\bcV_{\bH}$-based methods are not that far from the optimal positions. Additionally, $\bcV_{\bH}$-based method with error indicator, leads to smaller dimension, and hence, to larger acceleration rates.
In this example, $\bcV_{\bF}$-based methods show better results, although the results obtained by $\bcV_{\bH}$-based methods are also sufficiently good.

\begin{table}[H]
\begin{center}
\renewcommand{\arraystretch}{1.5}
\begin{tabular}{ |p{2.15cm}|p{1.90cm}|p{1.90cm}|p{2.15cm}| p{1.90cm}|p{2.65cm}|}\hline
& Original & $\bcV_{\bF}$ &$\bcV_{\bF}$ with $\bDelta(c)$ & $\bcV_{\bH}$ & $\bcV_{\bH}$ with $\bDelta(c)$\\\hline\hline
Time &	 $1.4\cdot 10^{4}$ & $1.4\cdot 10^{3}$ &  $6.91.4\cdot 10^{2}$ & $1.6\cdot 10^{3}$ & $6.0\cdot 10^{2}$\\\hline
Dimension & $601$ & $289$ & $208$ & $325$ & $200$\\\hline
Position &
$\begin{matrix}112,\;302,\\ 411\qquad\;\end{matrix}$ &
$\begin{matrix}112,\;302,\\ 411\qquad\;\end{matrix}$ &
$\begin{matrix}112,\;302,\\ 411\qquad\;\end{matrix}$ &
$\begin{matrix}114,\;303,\\ 411\qquad\;\end{matrix}$ &
$\begin{matrix}112,\;303,\\ 411\qquad\;\end{matrix}$\\\hline
Error & - & $0$ & $0$ & $6.1\cdot 10^{-3}$ &  $1.9\cdot 10^{-3}$\\\hline
Acceleration & - & $10$ & $22$ & $9$ & $24$\\\hline
\end{tabular}
\end{center}
\caption{Position optimization times and errors for the original and reduced optimization processes for the second example.}
\label{tab:Ex3_PosOnly}
\end{table}

Finally, in Table \ref{tab:Ex3}, we evaluate the optimization of both parameters, the damper's positions and the corresponding gains for the second example.
We observe that for all methods, the acceleration rates are not as large as for the other examples. In particular, for the case where $\bcV_{\bH}$ is approximated without using an error indicator, there is no acceleration at all, but we obtained the same results as for the original system.
Also we observe that other methods have difficulties to capture the behaviour of the original system.
The damper's positions are approximated sufficiently good.
However, the damping gains are not well approximated for these methods, while the rough area of optimal values are found.
%, but the specific values are not close enough to the real optima.
Hence, we see, that our methods have limits if the examples are too complicated.
In these cases the optima provide the approximated location of the external dampers.
Also, we can apply viscosity optimization methods as presented in \cite{morTomBG18,morBenKTetal16,morPrzPB23} to optimize the damper's viscosities for the positions obtained by our method.

\begin{table}[H]
\begin{center}
\renewcommand{\arraystretch}{1.5}
\begin{tabular}{ |p{2.15cm}|p{1.90cm}|p{1.90cm}|p{2.15cm}| p{1.90cm}|p{2.65cm}|}\hline
& Original & $\bcV_{\bF}$ &$\bcV_{\bF}$ with $\bDelta(c)$ & $\bcV_{\bH}$ & $\bcV_{\bH}$ with $\bDelta(c)$\\\hline\hline
Time &	 $4.3\cdot 10^{3}$ & $5.5\cdot 10^{2}$ &  $2.7\cdot 10^{2}$ &  $4.6\cdot 10^{3}$ & $8.4\cdot 10^{2}$\\\hline
Dimension & $601$ & $290$ & $208$ & $502$ & $325$\\\hline
Position &
$\begin{matrix}97,\;299,\\414\qquad\;\end{matrix}$ &
$\begin{matrix}89,\;296,\\411\qquad\;\end{matrix}$ &
$\begin{matrix}89,\;395,\\411\qquad\;\end{matrix}$ &
$\begin{matrix}97,\;299,\\414\qquad\;\end{matrix}$ &
$\begin{matrix}89,\;296,\\411\qquad\;\end{matrix}$\\\hline
Gain &
$\begin{matrix}8.3\cdot 10^{2},\\1.2\cdot 10^{3},\\4.6\cdot 10^{2}\end{matrix}$ &
$\begin{matrix}1.1\cdot 10^{3},\\1.3\cdot 10^{3},\\7.2\cdot 10^{2}\end{matrix}$ &
$\begin{matrix}1.1\cdot 10^{3},\\1.3\cdot 10^{3},\\7.2\cdot 10^{2}\end{matrix}$ &
$\begin{matrix}8.3\cdot 10^{2},\\ 1.2\cdot 10^{3},\\ 4.6\cdot 10^{2}\end{matrix}$ &
$\begin{matrix}1.1\cdot 10^{3},\\ 1.3\cdot 10^{3},\\ 7.2\cdot 10^{2}\end{matrix}$\\\hline
Error position & - & $1.9\cdot 10^{-2}$ & $2.1\cdot 10^{-2}$ & $0$ &  $1.9\cdot 10^{-2}$\\\hline
Error gain  & - & $2.6\cdot 10^{-1}$ & $2.6\cdot 10^{-1}$ & $0$ &  $2.6\cdot 10^{-1}$\\\hline
Acceleration & - & $8$ & $16$ & - & $5$\\\hline
\end{tabular}
\end{center}
\caption{Position and viscosity optimization times and errors for the original and reduced optimization processes for the second example.}
\label{tab:Ex3}
\end{table}

\section{Conclusions}
This paper has addressed the optimization of the positions and the viscosities of external dampers in vibrational systems.
For that, we have evaluated the corresponding system response that includes the computation of the controllability Gramian.
To accelerate the optimization process, we have decomposed the controllability space into a parameter-independent subspace and several parameter-dependent subspaces that have been used to derive an adaptive reduced basis method. This method determines a reduced basis of a subspace that approximates the controllability space for the required parameters.
Using this basis we have computed approximations of the Gramians and the corresponding system response values.
To assess the quality of the system response approximations, we have derived an error indicator that is based on the controllability space decomposition.
Finally, the efficiency of the methods has been illustrated for two examples.

\bibliographystyle{plainurl}
\bibliography{csc,mor,software,References}

\end{document}